\documentclass{article}
\usepackage{amsmath}
\usepackage{amsfonts}
\usepackage{amsthm}
\usepackage{amssymb, setspace}
\usepackage{mathtools}
\usepackage{lipsum}
\usepackage{color,graphicx,epsfig,geometry,fancyhdr,hyperref}
\usepackage[T1]{fontenc}
\usepackage{float}
\usepackage{subfig}
\usepackage{commath}
\usepackage{bbm}
\usepackage{bm}
\usepackage{mathrsfs,fleqn}
\usepackage{pgfplots}
\pgfplotsset{compat=newest}
\newtheorem{theorem}{Theorem}[section]

\newtheorem{example}{Example}[section]
\newcommand{\N}{\mathbb{N}}
\newcommand{\Z}{\mathbb{Z}}

\newcommand{\R}{\mathbb{R}}
\newcommand{\C}{\mathbb{C}}

\newcommand{\grad}{\nabla}

\usepackage{comment}
\usepackage{ulem}
\usepackage{stmaryrd}
\graphicspath{{paper-pics/}}

\begin{document}

\begin{flushleft}
\Large 
\noindent{\bf \Large  Two direct sampling methods for an anisotropic scatterer with a conductive boundary}
\end{flushleft}

\vspace{0.2in}
{\bf  \large Victor Hughes and Isaac Harris}\\
\indent {\small Department of Mathematics, Purdue University, West Lafayette, IN 47907 }\\
\indent {\small Email: \texttt{vhughes@purdue.edu}, \texttt{harri814@purdue.edu}\\

{\bf  \large Andreas Kleefeld}\\
\indent {\small Forschungszentrum J\"{u}lich GmbH, J\"{u}lich Supercomputing Centre, } \\
\indent {\small Wilhelm-Johnen-Stra{\ss}e, 52425 J\"{u}lich, Germany}\\
\indent {\small University of Applied Sciences Aachen, Faculty of Medical Engineering and } \\
\indent {\small Technomathematics, Heinrich-Mu\ss{}mann-Str. 1, 52428 J\"{u}lich, Germany}\\
\indent {\small Email: \texttt{a.kleefeld@fz-juelich.de}}\\

\begin{abstract}
\noindent In this paper, we consider the inverse scattering problem associated with an anisotropic medium with a conductive boundary condition. We will assume that the corresponding far--field pattern or Cauchy data is either known or measured. The conductive boundary condition models a thin coating around the boundary of the scatterer. We will develop two direct sampling methods to solve the inverse shape problem by numerically recovering the scatterer. To this end, we study direct sampling methods by deriving that the corresponding imaging functionals decay as the sampling point moves away from the scatterer. These methods have been applied to other inverse shape problems, but this is the first time they will be applied to an anisotropic scatterer with a conductive boundary condition. These methods allow one to recover the scatterer by considering an inner--product of the far--field data or the Cauchy data. Here, we will assume that the  Cauchy data is known on the boundary of a region $\Omega$ that completely encloses the scatterer $D$. We present numerical reconstructions in two dimensions to validate our theoretical results for both circular and non-circular scatterers.
\end{abstract}

\section{Introduction} \label{aicbc-section-intro}
We will study the {\it inverse shape problem} of recovering an unknown scatterer from measured far--field pattern or Cauchy data. For this inverse shape problem, we analyze the multiple direct sampling methods for reconstructing an anisotropic medium with a conductive boundary condition. We study the application of these direct sampling methods for an anisotropic material with a conductive boundary on an unbounded domain. Similar techniques have been used for other scattering problems, but this problem requires original analysis to complete. Here, the conductive boundary condition is modeled by a Robin condition that states that the total field has a jump across the boundary of the scatterer which is proportional to the co-normal derivative of the total field. The conductive boundary condition models a thin coating around the boundary of the scatterer. We remark that this is the same model as presented in \cite{new-conductive-model}. Precisely, in \cite{new-conductive-model}, the factorization method for recovering the scatterer from far--field measurements has been studied. Here, we study direct sampling methods to recover the scatterer which are more stable with respect to noisy data than the factorization method. Throughout our investigation, we need to derive a new factorization of the far--field operator. This factorization is based on the variational formulation of the scattering problem whereas in \cite{new-conductive-model} the factorization was based on integral operators. 

One of our main contributions is to extend the work performed in \cite{new-conductive-model} (see \cite{fmconductbc} for the isotropic case). This is accomplished by considering a more general model to what is studied in \cite{new-conductive-model}. In this paper, we will assume that there is an anisotropic coefficient as well as a non--trivial refractive index. Since the previous works did not consider a scatterer with a refractive index, the analysis presented here gives the factorization of the far--field operator with both the anisotropic coefficient and refractive index. Direct sampling type methods have been studied for many years \cite{DSM-DOT,DSM-EIT,multifreq-dsm,harris-Biharmonic-dsm,harris-nguyen,Ito2012,Ito2013}. The work in \cite{Liu} was the first case of analysis associated with the factorization method being used to analyze a direct sampling method. This idea was further developed in \cite{harris-dsm,harris-nguyen} where the authors gave explicit estimates that imply that the imaging functional decays (rapidly) as the sampling point moves away from the scatterer. This implies that one can simply plot the imaging functional which is defined via an inner--product of the data and specific test function(s). This provides a computationally stable method for approximating the shape of the scatterer with little a priori information.

Another contribution is to extend the applicability of the imaging functional studied in \cite{liem-paper-rgf} (see also \cite{liem2})  for the case of a medium with a conductive boundary and non--trivial refractive index. The direct sampling method used in \cite{liem-paper-rgf} uses the reciprocity gap functional to derive a computationally simple yet analytically rigorous imaging functional. The reciprocity gap functional has been used to solve many inverse shape problems, for e.g. \cite{Govanni-DSM,multifreq-dsm,dsm-maxwel-layer,RTM-gibc}. Recently in \cite{liem2,Govanni-DSM}, direct sampling methods derived from the reciprocity gap functional have been studied to give fast and accurate shape reconstructions. Here, for the case of far--field or Cauchy data, we will assume that we have full--aperture data, but in \cite{ayalaharriskleefeldex2024,Kang1,Kang2} the researchers have studied how to deal with partial--aperture data.

The rest of the paper is organized as follows: in Section \ref{acbc2-section-formulation}, we formulate the direct and inverse scattering problem for the given anisotropic material with a conductive boundary. Next, in Section \ref{aicbc-section-farfield}, we decompose the far--field operator into a symmetric factorization. This is necessary for the use of the direct sampling method via the far--field data. We also provide an explanation on the application of the far--field direct sampling method for this problem and its associated imaging functional. In Section \ref{aicbc-section-cauchy}, we derive a direct sampling method using the Cauchy data by studying the reciprocity gap functional for the given material. Finally, in Section \ref{aicbc-section-numerics}, we provide numerical examples in two dimensions for various scatterers to validate our theoretical results.

\section{Formulation of the Problem} \label{acbc2-section-formulation}
We start by discussing the direct scattering problem for the anisotropic material. We let $D \subset \mathbb{R}^d$ for $d=$ 2 or 3 be a simply connected open set with Lipschitz boundary $\partial D$, where $\nu$ represents the unit outward normal vector. The region $D$ denotes the (unknown) scatterer that we illuminate with an incident plane wave $u^i:=\text{e} ^{\text{i} kx\cdot\hat{y}}$ such that $k>0$ is the wave number, $x\in \mathbb{R}^d$, and $\hat{y} \in \mathbb{S}^{d-1}$ denotes the incident direction. $\mathbb{S}^{d-1}$ is defined to be the unit ball in $d-1$ dimensions, i.e. $\mathbb{S}^{d-1}:=\{x\in\mathbb{R}^d : |x|=1 \}$. We assume that we have a symmetric matrix--valued function $A(x) \in \mathscr{C}^1 (\overline{D}, \mathbb{C}^{d\times d})$ that satisfies
\begin{center}
    $ \overline{\xi} \cdot \text{Re}(A(x))\xi \geq A_{\text{min}}|\xi|^2 $ \quad and \quad $\overline{\xi} \cdot \text{Im}(A(x))\xi \leq 0$ \quad for a.e. $x \in D$ and $\xi \in \mathbb{C}^d$
\end{center}
where $A_{\text{min}}$ is a positive constant. We also assume that the refractive index $n(x) \in L^\infty (\mathbb{R}^d)$ satisfies
\begin{center}
    $\text{Im}(n(x)) \geq 0$ \quad for a.e. $x \in D$.
\end{center}
Outside of the scatterer $D$, the material parameters $A(x)$ and $n(x)$ are homogeneous and isotropic such that 
$$A(x)=I \quad \text{and} \quad n(x)=1\quad \text{for }x \in \R^d \setminus \overline{D}. $$
To model the thin conductive layer on the boundary, we consider the real--valued conductivity parameter $\eta \in L^\infty(\partial D)$, with the condition
\begin{center}
    $\eta(x) \geq \eta_{\text{min}} > 0$ \quad for a.e. $x \in \partial D$.
\end{center}

With the above assumptions, we have that the direct scattering problem for an anisotropic material with a conductive boundary condition is formulated as follows: find $u^s \in H^1_{loc}(\mathbb{R}^d \setminus \partial D)$ such that
\begin{align}
 \Delta u^s + k^2u^s = 0 \quad \text{in $\mathbb{R}^d \setminus \overline{D}$} \quad \text{ and } \quad  \nabla \cdot A\nabla u+k^2nu=0 \quad &\text{in $D$} \label{directaicbc1} \\ 
(u^s+u^i)_+ = u_- + \text{i}\eta (\nu\cdot A\nabla u_-)  \quad \text{ and } \quad  \partial_\nu (u^s+u^i)_+ = \nu \cdot A\nabla u_-  \quad &\text{on $\partial D$}  \label{directaicbc2} 
\end{align}
where $u = u^s+u^i$ denotes the total field. Here, the subscripts $+$ and $-$ denote approaching the boundary from the outside and inside of $D$, respectively. Even though \eqref{directaicbc1}--\eqref{directaicbc2} is written in terms of $u^s$ in $\mathbb{R}^d \setminus \overline{D}$ and $u$ in $D$, since $u = u^s+u^i$ we consider the boundary values problem as a problem for just $u^s$. We also assume that the scattered field satisfies the radiation condition 
\begin{align}
    \partial_r u^s -\text{i} ku^s = \mathcal{O}\Big(\frac{1}{r^{(d+1)/2}}\Big)  \quad \text{ as } \quad r \to \infty \label{Sommerfield_aicbc}
\end{align}
uniformly with respect to $\hat{x}={x}/{r}$ where $r=|x|$. With the above assumptions on the coefficients, we have well--posedness of the direct problem \eqref{directaicbc1}--\eqref{directaicbc2} along with the radiation condition by appealing to similar as analysis in \cite[Section 2]{new-conductive-model} (see also \cite{cakonicoltonmonk05}). By appealing to \cite[Theorem 2.1]{new-conductive-model} along with the Trace Theorems, we have that 
$$\| u^s \|_{H^1(B_R \setminus \partial D)} \leq C \| u^i \|_{H^1(D)}.$$
Here, $B_R$ denotes the ball centered at the origin with radius $R>0$ such that $D \subset B_R$. 

We are interested in recovering the scatterer $D$ from measured far--field and Cauchy data. To this end, recall that the asymptotic expansion of the scattering field $u^s(\cdot \, , \hat{y})$ has the form
\begin{align*}
    u^s(x,\hat{y}) = \gamma_d\frac{\text{e}^{\text{i}k|x|}}{|x|^{\frac{d-1}{2}}} \Bigg\{  u^\infty(\hat{x}, \hat{y}) + \mathcal{O}\Bigg( \frac{1}{|x|} \Bigg)   \Bigg\}
\end{align*}
where $\gamma_d$ is given by 
$$\gamma_2 = \frac{\text{e}^{\text{i}\pi /4}}{\sqrt{8\pi k}}\quad \text{and} \quad \gamma_3=\frac{1}{4\pi}.$$ 
Here, $u^\infty(\hat{x},\hat{y})$ is the corresponding far--field pattern associated with \eqref{directaicbc1}--\eqref{directaicbc2}, which is dependent on the incident direction $\hat{y}$ and the observation direction $\hat{x}$. 
Given the far--field pattern, we define the far--field operator $F: L^2(\mathbb{S}^{d-1}) \longrightarrow L^2(\mathbb{S}^{d-1})$ as
\begin{align*}
    (Fg)(\hat{x}) := \int_{\mathbb{S}^{d-1}} u^\infty (\hat{x}, \hat{y})g(\hat{y}) \,  \text{d}s(\hat{y}) \quad \text{ where } \quad g \in L^2(\mathbb{S}^{d-1}).
\end{align*}
It is well known that the far--field operator $F$ is compact, see \cite{new-conductive-model}. For the Cauchy data, we will assume that 
$$\Big\{u^s(\cdot \, , \hat{y}) \, ,\, \partial_\nu u^s(\cdot \, , \hat{y})  \Big\}\Big|_{\partial \Omega} \quad \text{for all} \quad \hat{y} \in \mathbb{S}^{d-1} $$
where $\overline{D} \subset \Omega$. $\Omega \subset \mathbb{R}^d$ is a simply connected open set with a Lipschitz boundary $\partial\Omega$.

In order to do our analysis in the next section, we need to have a new expression for the far--field operator $F$. To this end, note that the fundamental solution to the Helmholtz Equation satisfies
$$\Delta \Phi(x,\cdot) +k^2\Phi(x,\cdot)=-\delta(x-\cdot) \quad \text{in } \mathbb{R}^d$$
along with the radiation condition and is given by 
\begin{align}
    \Phi(x,y) = \begin{cases}
        \frac{\text{i}}{4} H_0^{(1)}(k|x-y|) & \quad \text{in } \mathbb{R}^2, \\
        \\
        \frac{\text{e}^{\text{i}k|x-y|}}{4\pi |x-y|} & \quad \text{in }  \mathbb{R}^3.
    \end{cases} \label{aicbc-fundamentalsol-def}
\end{align}
Using Green's 2nd Identity in the interior of $D$ with the scattered field and fundamental solution gives us
\begin{align*}
    u^s(x)\chi_D = -\int_D &\Phi(x,z)[ \Delta u^s(z) +k^2u^s(z) ] \,\text{d}z + \int_{\partial D} \Phi(x,z)\partial_\nu u^s_-(z) - u^s_-(z)\partial_\nu \Phi(x,z) \,\text{d}s(z)  
\end{align*}
where $\chi_D = \chi_D(x)$ is the indicator function on the scatterer $D$ at $x\in \mathbb{R}^d$. Similarly, using Green's 2nd Identity in the exterior of $D$ and using that the scattered field solves Helmholtz equation on the exterior of $D$, we have that 
\begin{align*}
    u^s(x)(1-\chi_D) = \int_{\partial B_R} \Phi(x,z)\partial_\nu u^s(z) - u^s(z)\partial_\nu \Phi(x,z) \,\text{d}s(z) - \int_{\partial D} \Phi(x,z)\partial_\nu u^s_+(z) - u^s_+(z)\partial_\nu \Phi(x,z) \,\text{d}s(z)
\end{align*}
where again $B_R \subset \mathbb{R}^d$ is a ball of radius $R$ such that $D \subset B_R $. Therefore, if we add the above expressions, use Green's 1st Identity, and use the following equivalent expression of (\ref{directaicbc1})--(\ref{directaicbc2})
\begin{align*}
    \Delta u^s+k^2u^s=\nabla\cdot (I-A)\nabla(u^s+u^i) +k^2(1-n)(u^s+u^i) \quad &\text{in $\mathbb{R}^d \setminus \partial D$} \\
    u^s_+-u^s_- = \text{i} \eta (\nu\cdot A\nabla (u^s_-+u^i))  \quad \text{ and } \quad  \partial_\nu (u^s_+ +u^i) = \nu \cdot A\nabla (u^s_-+u^i)  \quad &\text{on $\partial D$} ,
\end{align*}
we will get a Lippmann--Schwinger representation of the scattered field
\begin{align}
    u^s(x) = \int_D &\nabla \Phi(x,z) \cdot (I-A)\nabla \big(u^s(z)+u^i(z)\big) + k^2(n-1)\big(u^s(z)+u^i(z)\big)\Phi(x,z) \,\text{d}z \nonumber \\
    &\quad \quad \quad + \int_{\partial D} \text{i}\eta \partial_\nu \big(u_+^s(z)+u^i(z)\big)\partial_\nu \Phi(x,z) \,\text{d}s(z).
\end{align}
Here, we used the fact that the fundamental solution and the scattered field satisfy the radiation condition (\ref{Sommerfield_aicbc}) which implies that the boundary integral on $\partial B_R$ tends to zero as $R\longrightarrow \infty$. This gives that the far--field pattern is given by the expression
\begin{align}
    u^\infty(\hat{x},\hat{y})  = \int_D &\nabla \text{e}^{-\text{i}k\hat{x}\cdot z} \cdot (I-A)\nabla \big(u^s(z)+u^i(z)\big) + k^2(n-1)\big(u^s(z)+u^i(z)\big)\text{e}^{-\text{i}k\hat{x}\cdot z} \,\text{d}z \nonumber \\
    &\quad + \int_{\partial D} \text{i}\eta \partial_\nu \big(u_+^s(z)+u^i(z)\big)\partial_\nu \text{e}^{-\text{i}k\hat{x}\cdot z} \,\text{d}s(z).
\end{align}
Using this formula for the far--field pattern, we obtain the following expression for the far--field operator by interchanging the order of integration 
\begin{align}
    Fg = \int_D &\nabla \text{e}^{-\text{i}k\hat{x}\cdot z} \cdot (I-A)\nabla (u^s_g(z)+v_g(z)) + k^2(n-1)(u^s_g(z)+v_g(z))\text{e}^{-\text{i}k\hat{x}\cdot z} \,\text{d}z \nonumber \\
    &\quad+ \int_{\partial D} \text{i}\eta \partial_\nu (u_{g,+}^s(z)+v_g(z))\partial_\nu \text{e}^{-\text{i}k\hat{x}\cdot z} \,\text{d}s(z). \label{aicbc-farfieldoperator-def}
\end{align}
Here, we let $v_g$ denote the Hergoltz wave function and $u_g^s$ solves the boundary value problem (\ref{directaicbc1})--(\ref{Sommerfield_aicbc}) with incident field $u^i=v_g$, defined as
\begin{align*}
    v_g(x) = \int_{\mathbb{S}^{d-1}} \text{e}^{\text{i}kx\cdot \hat{y}}g(\hat{y}) \,  \text{d}s(\hat{y})   \quad \text{ and } \quad u^s_g(x) = \int_{\mathbb{S}^{d-1}} u^s(x, \hat{y})g(\hat{y}) \,  \text{d}s(\hat{y}).
\end{align*}
We are interested in the inverse problem of recovering $D$ from the given the far--field and Cauchy data. The main goal is to study multiple direct sampling methods to perform these reconstructions of $D$, given we have the far--field and Cauchy data. In the next section, we will analyze the direct sampling method via the far--field data.

\section{Reconstruction via Far--Field Data} \label{aicbc-section-farfield}
In order to utilize the direct sampling method first considered in \cite{Liu} for reconstructing the region $D$ using the far--field data, we first need to obtain a symmetric factorization of the far--field operator. The factorization along with the Funk--Hecke integral identity will allow us to derive the resolution analysis of our imaging functional as in \cite{ayalaharriskleefeld2024,harris-dsm}. Just as in the previously cited works, we first consider the imaging functional 
$$I(z):= \big| \big(F\phi_z,\phi_z \big)_{L^2(\mathbb{S}^{d-1})}\big|^p$$ 
where $p \in \N$ and $\phi_z = \text{e}^{- \text{i}k\hat{x}\cdot z}$ i.e. the conjugate of the plane wave. Here, the parameter $p$ is chosen to sharpen the reconstruction of the scatterer. We will show that by plotting $I(z)$ we can recover the scatterer. This is due to the fact that $I(z)$ will decay sharply as the sampling point $z$ moves away from the boundary. Similar analysis was used for similar methods in \cite{Ito2012,Ito2013,liem,liem2}.

Now, given the expression for the far--field operator in (\ref{aicbc-farfieldoperator-def}), we derive a symmetric factorization for the far--field operator $F$. To this end, we first consider the following product space
\begin{align*}
   X:= L^2(D,\mathbb{C}^d)\times L^2(D)\times H^{1/2}(\partial D) 
\end{align*}
and its dual space, given by
\begin{align*}
   X^*:= L^2(D,\mathbb{C}^d)\times L^2(D)\times H^{-1/2}(\partial D)  .
\end{align*}
The sesquilinear dual product $\big\langle \cdot \, , \cdot \big\rangle_{X,X^*}$ is defined to be
\begin{align*}
    \big\langle (\phi_1,\phi_2,\phi_3), (\psi_1,\psi_2,\psi_3) \big\rangle_{X,X^*} = \int_D \phi_1\cdot\overline{\psi_1} + \phi_2 \overline{\psi_2} \,\text{d}x\ + \int_{\partial D} \phi_3 \overline{\psi_3} \,\text{d}s\ ,
\end{align*}
for $(\phi_1,\phi_2,\phi_3) \in X$ and $(\psi_1,\psi_2,\psi_3) \in X^*$. To get an initial factorization for $F$, we first need to define the source to far--field operator $G: X^* \longrightarrow L^2(\mathbb{S}^{d-1})$, given by
\begin{align}
    G(f,p,h) = \int_D &\nabla \text{e}^{-\text{i}k\hat{x}\cdot z} \cdot (I-A)(\nabla w+f) + k^2(n-1)(w+p)\text{e}^{-\text{i}k\hat{x}\cdot z} \,\text{d}z \nonumber \\
    &+ \int_{\partial D} \text{i}\eta (\partial_\nu w_{+} +h)\partial_\nu \text{e}^{-\text{i}k\hat{x}\cdot z} \,\text{d}s  \label{def-of-G}
\end{align}
where  $G(f,p,h) = w^\infty$ is the far--field pattern for $w \in H^1_{loc}(\mathbb{R}^d \setminus \partial D)$ that is the unique solution to the auxiliary problem
\begin{align}
    \nabla \cdot A\nabla w+k^2n w=\nabla\cdot (I-A)f +k^2(1-n)p \quad &\text{in $\mathbb{R}^d \setminus \partial D$} \label{directaicbc_aux1} \\ 
    w_+-w_- = \text{i}\eta (\partial_\nu w_+ + h)  \quad \text{ and } \quad  \partial_\nu w_+  = \nu \cdot A\nabla w_- +\nu\cdot(A-I)f  \quad &\text{on $\partial D$}  \label{directaicbc_aux2} 
\end{align}
along with the radiation condition for $w$, where $(f,p,h) \in X^*$. Note that the well--posedness of \eqref{directaicbc_aux1}--\eqref{directaicbc_aux2} with the radiation condition can be established just as in \cite{new-conductive-model}. Notice that the scattered field $u^s$ solves (\ref{directaicbc_aux1})--(\ref{directaicbc_aux2}) with 
$$(f,p,h) = \Big( \nabla u^i |_D, u^i |_D, \partial_\nu u^i |_{\partial D} \Big).$$ 
After some calculations, one can show that the solution $w$ to (\ref{directaicbc_aux1})--(\ref{directaicbc_aux2}) satisfies the variational identity 
\begin{align}
    -\int_{B_R} \nabla w\cdot\nabla\overline{\phi}-k^2w\overline{\phi} \,\text{d}x\ +\int_{\partial B_R}\overline{\phi}\Lambda_k w\,\text{d}s = \int_D &\nabla\overline{\phi}\cdot (A-I)(\nabla w+f) - k^2(n-1)(w+p)\overline{\phi} \,\text{d}x\ \nonumber \\
    &- \int_{\partial D} \frac{\text{i}}{\eta}\overline{\llbracket \phi \rrbracket} \llbracket w \rrbracket + \overline{\llbracket \phi \rrbracket} h \,\text{d}s\ \label{aicbc-aux-varform1}
\end{align}
for any $\phi\in H^1(B_R \setminus \partial D)$. This is obtained by multiplying the PDE by $\phi$ and using Green's Identities. In the above equation, we define $\llbracket \phi \rrbracket := (\phi_{+} - \phi_{-})$ i.e. the jump in the trace on $\partial D$. Here, we let  
$$\Lambda_k: H^{1/2}(B_R) \longrightarrow H^{-1/2}(B_R)$$ 
denote the Dirichlet-to-Neumann map on $\partial B_R$ for the radiating solution to the Helmholtz equation on the exterior of $B_R$, defined as 
\begin{equation*}
    \Lambda_k f = \partial_\nu \varphi \big|_{\partial B_R} \quad \text{ where }\quad  \Delta \varphi+k^2 \varphi=0 \text{ in $\mathbb{R}^d \setminus \overline{B_R}$} \quad \text{ and } \quad \varphi\big|_{\partial B_R} = f, 
\end{equation*}
along with the radiation condition (\ref{Sommerfield_aicbc}). From \cite[Theorem 5.22]{Cakoni-Colton-book} we know that $\Lambda_k$ is a bounded linear operator. This is a direct consequence of the well--posedness of the above Dirichlet problem along with the (Neumann) Trace Theorem. Next, we define the bounded linear operator 
\begin{align}
    H: L^2(\mathbb{S}^{d-1}) \longrightarrow X^* \quad \text{ given by }\quad (Hg)(x) = \Big( \nabla v_g |_D, v_g|_D, \partial_\nu v_g|_{\partial D} \Big). \label{aicbc_Hdef} 
\end{align}
By the superposition principle, we have that the far--field operator associated with (\ref{directaicbc1})--(\ref{Sommerfield_aicbc}) is given by $F=GH$. 

In order to get the symmetric factorization for $F$, we will now compute the adjoint of the operator $H$ defined in \eqref{aicbc_Hdef}. Notice that
\begin{align*}
    \big\langle  (\phi_1,\phi_2,\phi_3) , Hg \big\rangle_{X,X^*} &= \big( H^*(\phi_1,\phi_2,\phi_3) , g \big) _{L^2(\mathbb{S}^{d-1})} \\
    &= \int_D \nabla \overline{v_g} \cdot {\phi_1} + \overline{v_g} {\phi_2} \,\text{d}x + \int_{\partial D} {\phi_3}\partial_\nu \overline{v_g} \,\text{d}s \\
    &= \int_{\mathbb{S}^{d-1}} \overline{g(\hat{y})} {\left[  \int_D \nabla \text{e}^{-\text{i}kx\cdot\hat{y}} \cdot \phi_1 + \text{e}^{-\text{i}kx\cdot\hat{y}}\phi_2 \,\text{d}x + \int_{\partial D} \phi_3\partial_\nu \text{e}^{-\text{i}kx\cdot\hat{y}} \,\text{d}s(x)  \right]} \,\text{d}s(\hat{y}) .
\end{align*}
Therefore, we have that $H^*: X \longrightarrow L^2(\mathbb{S}^{d-1})$ is given by \\
$$H^*(\phi_1,\phi_2,\phi_3) = \int_D \nabla \text{e}^{-\text{i}kx\cdot\hat{y}} \cdot \phi_1 + \text{e}^{-\text{i}kx\cdot\hat{y}}\phi_2 \,\text{d}x + \int_{\partial D} \phi_3\partial_\nu \text{e}^{-\text{i}kx\cdot\hat{y}} \,\text{d}s(x) .$$
This leads us to define a middle operator $T:X^* \longrightarrow X$ as
\begin{align}
    T(f,p,h) = \Big( (I-A)(\nabla w+f)|_D ,\quad k^2(n-1)(w+p)|_D,\quad \text{i}\eta (\partial_\nu w_{+} +h)|_{\partial D} \Big) .\label{aicbc_Tdef}
\end{align}
Therefore, by appealing to the definition of $G$ in \eqref{def-of-G} we have that $G=H^*T$ and thus $F=H^*TH$, proving that we have a symmetric factorization for $F$. 
\begin{theorem} \label{aicbc-thm-symm_factorization}
    The far--field operator $F: L^2(\mathbb{S}^{d-1}) \longrightarrow L^2(\mathbb{S}^{d-1})$ has the symmetric factorization $F=H^*TH$, where $H$ and $T$ are defined in (\ref{aicbc_Hdef}) and (\ref{aicbc_Tdef}), respectively.
\end{theorem}

The symmetric factorization above is one of the main ingredients needed for studying the direct sampling method. In order to proceed, we need to study the middle operator $T$ in our factorization of $F$. Ultimately, we will show that $T$ is a coercive operator on $\overline{\mathrm{Range}(H)}$. We achieve this by proving that $T$ is the sum of a compact operator and a coercive operator. This approach is similar to the study of the so--called factorization method as seen in \cite[Lemma 1.17]{FM-Book}. Before we proceed, we need to define the real--part of an operator $S:X^* \longrightarrow X$. It is given by
$$\mathrm{Re}(S) = \frac{1}{2}\big(S+S^*\big)$$
where we assume that $X \subset X^*$. 

\begin{theorem} \label{aicbc-thm-T=S+K}
    Let the operator $T$ be defined as in (\ref{aicbc_Tdef}).
    \begin{enumerate}
        \item If $\mathrm{Re}(A)-I$ is positive definite uniformly in $D$, then $-T=S+K$ where $\mathrm{Re}(S)$ is a coercive operator and $K$ is a compact operator on $\overline{\mathrm{Range}(H)}$.
        \item If $I-\mathrm{Re}(A)-\alpha Im(A)$ and $\mathrm{Re}(A)-\frac{1}{\alpha}\mathrm{Im}(A)$ are positive definite uniformly in $D$, then $T=S+K$ where $\mathrm{Re}(S)$ is a coercive operator and $K$ is a compact operator on $\overline{\mathrm{Range}(H)}$.
    \end{enumerate}
\end{theorem}
\begin{proof}
    (1) First, we analyze the operator $-T$. For us to be able to define our operators $S$ and $K$, we need to provide a new expression for our operator $-T$. First, by using the definition (\ref{aicbc_Tdef}) we have
    \begin{align}
        \big\langle T(f,p,h),(\psi_1,\psi_2,\psi_3)\big\rangle _{X,X^*} = \int_D &\overline{\psi_1}\cdot (I-A)(\nabla w +f)+k^2(n-1)(w+p)\overline{\psi_2} \,\text{d}x\ \nonumber \\
        &+ \int_{\partial D}\text{i}\eta(\partial_\nu w_{+} +h)\overline{\psi_3} \,\text{d}s \label{aicbc-Tdef-part1}
    \end{align}
    Let $(f_j,p_j,h_j)$ solve the auxiliary problem (\ref{directaicbc_aux1})--(\ref{directaicbc_aux2}) and the variational identity (\ref{aicbc-aux-varform1}) with the solution $w_j$, for $j=1,2$. Now, consider the variational form with respect to $(w_1,f_1,p_1,h_1)$ with $\phi=w_2$, given as
    \begin{align}
        -\int_{B_R} \nabla w_1\cdot\nabla\overline{w_2}-k^2w_1\overline{w_2} \,\text{d}x\ &+\int_{\partial B_R}\overline{w_2}\Lambda_k w_1\,\text{d}s \nonumber \\ 
        &\hspace{-1.25in}= \int_D \nabla\overline{w_2}\cdot (A-I)(\nabla w_1+f_1) - k^2(n-1)(w_1+p_1)\overline{w_2} \,\text{d}x- \int_{\partial D} \frac{\text{i}}{\eta}\overline{\llbracket w_2 \rrbracket} \llbracket w_1 \rrbracket + \overline{\llbracket w_2 \rrbracket} h_1 \,\text{d}s . \label{aicbc-varform-with-u2}
    \end{align}
    Now, notice that 
    \begin{align*}
    \big\langle T(f_1&,p_1,h_1),(f_2,p_2,h_2)\big\rangle _{X,X^*} \\
    &= \int_D \overline{f_2}\cdot (I-A)(\nabla w_1 +f_1)+k^2(n-1)(w_1+p_1)\overline{p_2} \,\text{d}x + \int_{\partial D}\text{i}\eta(\partial_\nu w_{1,+}+h_1)\overline{h_2} \,\text{d}s \\
        &= \int_D (\overline{\nabla w_2+f_2})\cdot (I-A)(\nabla w_1 +f_1)+k^2(n-1)(w_1+p_1)(\overline{w_2+p_2}) \,\text{d}x \\
        & \quad \quad \quad - \int_D \overline{\nabla w_2}\cdot (I-A)(\nabla w_{1} +f_1)+k^2(n-1)(w_1+p_1)\overline{w_2} \,\text{d}x\ .
    \end{align*}
    Using the variational form (\ref{aicbc-varform-with-u2}), we can now get
    \begin{align*}
        \big\langle T(f_1,p_1,h_1),(f_2,p_2,h_2)\big\rangle _{X,X^*} &= \int_D (\overline{\nabla w_2+f_2})\cdot (I-A)(\nabla w_1 +f_1)+k^2(n-1)(w_1+p_1)(\overline{w_2+p_2}) \,\text{d}x\  \\
        &\quad \quad \quad+ \int_{\partial D}\text{i}\eta(\partial_\nu w_{1,+}+h_1)\overline{h_2}+ \frac{\text{i}}{\eta}\overline{\llbracket w_2 \rrbracket} \llbracket w_1 \rrbracket + \overline{\llbracket w_2 \rrbracket} h_1 \,\text{d}s \\
        &\quad \quad \quad -\int_{B_R} \nabla w_1\cdot\nabla\overline{w_2}-k^2w_1\overline{w_2} \,\text{d}x\ +\int_{\partial B_R}\overline{w_2}\Lambda_k w_1\,\text{d}s .
    \end{align*}
    After using the boundary condition $\llbracket w_j \rrbracket = \text{i}\eta(\partial_\nu w_{j,+} +h_j)$, for $j=1,2$, and canceling terms, we finally get that
    \begin{align}
        \big\langle -T(f_1,p_1,h_1),(f_2,p_2,h_2)\big\rangle _{X,X^*} &= \int_D (\overline{\nabla w_2+f_2})\cdot (A-I)(\nabla w_1 +f_1)-k^2(n-1)(w_1+p_1)(\overline{w_2+p_2}) \,\text{d}x\  \nonumber\\
        & - \int_{\partial D} \text{i}\eta(\partial_\nu w_{1,+}+h_1)(\overline{\partial_\nu w_{2,+}+ h_2}) \,\text{d}s\nonumber \\
        &+\int_{B_R} \nabla w_1\cdot\nabla\overline{w_2}-k^2w_1\overline{u_2} \,\text{d}x\ -\int_{\partial B_R}\overline{w_2}\Lambda_k w_1\,\text{d}s . \label{aicbc-negTdef}
    \end{align}
    Now, we can define $-T=S+K$ where the operators $S$ and $K$ are given as
    \begin{align}
        \big\langle S(f_1,p_1,h_1),(f_2,p_2,h_2)\big\rangle _{X,X^*} &= \int_D (\overline{\nabla w_2+f_2})\cdot (A-I)(\nabla w_1 +f_1)+(w_1+p_1)(\overline{w_2+p_2}) \,\text{d}x\  \nonumber\\
        &- \int_{\partial D} \text{i}\eta(\partial_\nu w_{1,+}+h_1)(\overline{\partial_\nu w_{2,+}+ h_2}) \,\text{d}s \nonumber\\
        &+\int_{B_R} \nabla w_1\cdot\nabla\overline{w_2}+w_1\overline{w_2} \,\text{d}x\ -\int_{\partial B_R}\overline{w_2}\Lambda_k w_1\,\text{d}s \label{aicbc-negT-Sdef}
    \end{align}
    and
    \begin{align}
        \big\langle K(f_1,p_1,h_1),(f_2,p_2,h_2)\big\rangle _{X,X^*} &= \int_D -k^2(n-1)(w_1+p_1)(\overline{w_2+p_2}) - (w_1+p_1)(\overline{w_2+p_2}) \,\text{d}x \nonumber \\
        &- \int_{B_R} (k^2+1)w_1\overline{w_2}\,\text{d}x . \label{aicbc-negT-Kdef}
    \end{align}
    Notice, by virtue of the compact embeddings $H^1(D)$ into $L^2(D)$ and $H^1(B_R)$ into $L^2(B_R)$ we have that $K$ is a compact operator on $\overline{\text{Range}(H)}$.
    To complete the proof, we will show that $\mathrm{Re}(S)$ is a coercive operator provided $\text{Re}(A)-I$ is positive definite uniformly in $D$. To this end, notice that
    \begin{align}
        \big\langle \mathrm{Re}(S)(f_1,p_1,h_1),(f_1,p_1,h_1)\big\rangle _{X,X^*} &= \int_D (\overline{\nabla w_1+f_1})\cdot (\mathrm{Re}(A)-I)(\nabla w_1 +f_1)+|w_1+p_1|^2 \,\text{d}x\  \nonumber\\
        &+\int_{B_R} |\nabla w_1|^2+|w_1|^2 \,\text{d}x\ -\int_{\partial B_R}\overline{w_1}\mathrm{Re}(\Lambda_k) w_1\,\text{d}s\ \label{aicbc-negT-Re(S)def}
    \end{align}
    where $\mathrm{Re}(\Lambda_k)$ is a non-positive operator (see for e.g. \cite{CCH-book}). We want to show that $\mathrm{Re}(S)$ is positive coercive, i.e. there exists a $\beta >0$ independent of $(f_1,p_1,h_1)\in X^*$ such that
    $$\big\langle \mathrm{Re}(S)(f_1,p_1,h_1),(f_1,p_1,h_1)\big\rangle _{X,X^*} \geq \beta\|(f_1,p_1,h_1)\|^2_{X^*} .$$
    To this end, we proceed with a contradiction argument. We assume $\text{Re}(S)$ has no such constant, meaning there exists a sequence $(f^n,p^n,h^n)\in X^*$ with corresponding $w^n\in H^1_{loc}(\mathbb{R}^d\setminus \partial D)$ satisfying (\ref{directaicbc_aux1})--(\ref{directaicbc_aux2}) such that
    $$\|(f^n,p^n,h^n)\|^2_{X^*}=1 \quad\text{ and }\quad \big\langle \mathrm{Re}(S)(f^n,p^n,h^n),(f^n,p^n,h^n)\big\rangle _{X,X^*}\leq \frac{1}{n}.$$
    Since $\mathrm{Re}(A)-I$ is positive definite and $\mathrm{Re}(\Lambda_k)$ is non-positive, then as $n\longrightarrow \infty$, we have
    $$\big\langle \mathrm{Re}(S)(f^n,p^n,h^n),(f^n,p^n,h^n)\big\rangle _{X,X^*} \longrightarrow 0 \quad \text{implying}\quad w^n\longrightarrow 0 \quad \text{in }H^1(B_R \setminus \partial D) \quad \text{implying}$$
    $$f^n \longrightarrow 0 \quad\text{in $[L^2(D)]^d$ and}\quad p^n \longrightarrow 0 \quad \text{in $L^2(D)$}$$
    therefore, via the auxillary boundary value problem variational form, we also have that $h^n \longrightarrow 0$ in $H^{-1/2}(\partial D)$. This contradicts the normalization of $(f^n,p^n,h^n)$ in $X^*$. Therefore, $\mathrm{Re}(S)$ is a positive coercive operator provided $\mathrm{Re}(A)-I$ is a positive definite matrix. \\

    (2) We will now have new assumptions on the coefficients, so we need to derive a new variational expression for the operator $T$. To this end, consider again $(f_j,p_j,h_j)$ solving the auxiliary problem (\ref{directaicbc_aux1})--(\ref{directaicbc_aux2}) and the variational form (\ref{aicbc-aux-varform1}) with the solution $w_j$, for $j=1,2$. Now, from (\ref{aicbc-Tdef-part1}), notice that
    \begin{align}
        \big\langle T(f_1,p_1,h_1)&,(f_2,p_2,h_2)\big\rangle _{X,X^*} \nonumber \\
         &= \int_D \overline{f_2}\cdot (I-A)(\nabla w_1 +f_1)+k^2(n-1)(w_1+p_1)\overline{p_2} \,\text{d}x + \int_{\partial D} \text{i}\eta(\partial_\nu w_{1,+}+h_1)\overline{h_2} \,\text{d}s \nonumber \\ 
        &= \int_D \overline{f_2}\cdot (I-A)f_1+k^2(n-1)p_1\overline{p_2} \,\text{d}x\  + \int_{\partial D} \text{i}\eta(\partial_\nu w_{1,+}+h_1)\overline{h_2} \,\text{d}s \nonumber \\ 
        &\quad \quad \quad + \int_D \overline{f_2}\cdot (I-A)\nabla w_1+k^2(n-1)w_1\overline{p_2} \, \text{d}x. \nonumber
    \end{align}
    Now, consider the following variational form with respect to $(w_2,f_2,p_2,h_2)$ with $\phi=w_1$ and conjugate both sides, we obtain
    \begin{align}
        \int_{B_R} \overline{A}\nabla w_1\cdot\nabla\overline{w_2}&-k^2\overline{n}w_1\overline{w_2} \,\text{d}x\ -\int_{\partial B_R}w_1\overline{\Lambda_k w_2}\,\text{d}s \nonumber \\ 
        &= \int_D \overline{f_2}\cdot (I-\overline{A})\nabla w_1 + k^2(\overline{n}-1)w_1\overline{p_2} \,\text{d}x+ \int_{\partial D} -\frac{ \text{i}}{\eta}\overline{\llbracket w_2 \rrbracket} \llbracket w_1 \rrbracket + \llbracket w_1 \rrbracket \overline{h_2} \,\text{d}s . \label{aicbc-varform-with-u1}
    \end{align}
    Substituting this variational form into the definition of $T$, we can rewrite the definition of $T$ to obtain
    \begin{align*}
        \big\langle T(f_1,p_1,h_1),(f_2,p_2,h_2)\big\rangle _{X,X^*} &= \int_D \overline{f_2}\cdot (I-A)f_1+k^2(n-1)p_1\overline{p_2} \,\text{d}x\ + \int_{\partial D} \text{i}\eta(\partial_\nu w_{1,+}+h_1)\overline{h_2} \,\text{d}s \\
        &\hspace{-0.5in}+ \int_{\partial D} \frac{ \text{i}}{\eta}\overline{\llbracket w_2 \rrbracket} \llbracket w_1 \rrbracket - \llbracket w_1 \rrbracket \overline{h_2} \,\text{d}s + \int_D \overline{f_2}\cdot (\overline{A}-A)\nabla w_1+k^2(n-\overline{n})w_1\overline{p_2} \,\text{d}x\ \nonumber \\
        &\hspace{-0.5in}+\int_{B_R} \overline{A}\nabla w_1\cdot\nabla\overline{u_2}-k^2\overline{n}w_1\overline{u_2} \,\text{d}x\ -\int_{\partial B_R}w_1\overline{\Lambda_k w_2}\,\text{d}s .
    \end{align*}
    Using the boundary condition $\llbracket w_j \rrbracket = \text{i}\eta(\partial_\nu w_{j,+} +h_j)$, for $j=1,2$, and canceling terms, we also have
    \begin{align*}
        \big\langle T(f_1,p_1,h_1),(f_2,p_2,h_2)\big\rangle _{X,X^*} &= \int_D \overline{f_2}\cdot (I-A)f_1+k^2(n-1)p_1\overline{p_2} \,\text{d}x\ + \int_{\partial D} \frac{ \text{i}}{\eta}\llbracket w_1 \rrbracket\overline{\llbracket w_2 \rrbracket}  \,\text{d}s\ \nonumber \\ 
        &- 2 \text{i}\int_D \overline{f_2}\cdot \mathrm{Im}(A)\nabla w_1-k^2 \mathrm{Im}(n)w_1\overline{p_2} \,\text{d}x\ \nonumber \\
        &+\int_{B_R} \overline{A}\nabla w_1\cdot\nabla\overline{u_2}-k^2\overline{n}w_1\overline{w_2} \,\text{d}x\ -\int_{\partial B_R}w_1\overline{\Lambda_k w_2}\,\text{d}s.
    \end{align*}
    We can now define $T=S+K$ where the operators $S$ and $K$ are given by
    \begin{align}
        \big\langle S(f_1,p_1,h_1)&,(f_2,p_2,h_2)\big\rangle _{X,X^*} \nonumber\\
        &= \int_D \overline{f_2}\cdot (I-A)f_1+p_1\overline{p_2} \,\text{d}x\ + \int_{\partial D} \frac{ \text{i}}{\eta}\llbracket w_1 \rrbracket\overline{\llbracket w_2 \rrbracket}  \,\text{d}s - 2 \text{i}\int_D \overline{f_2}\cdot \mathrm{Im}(A)\nabla w_1 \,\text{d}x\ \nonumber \\
        &+\int_{B_R} \overline{A}\nabla w_1\cdot\nabla\overline{w_2} + w_1\overline{w_2}  \,\text{d}x\ -\int_{\partial B_R}w_1\overline{\Lambda_k w_2}\,\text{d}s\ \label{aicbc-posT-Sdef}
    \end{align}
    and
    \begin{align}
        \big\langle K(f_1,p_1,h_1),(f_2,p_2,h_2)\big\rangle _{X,X^*} &=  \int_D k^2(n-1)p_1\overline{p_2}-p_1\overline{p_2} \,\text{d}x\ - \int_{B_R} (k^2+1)\overline{n}w_1\overline{w_2}\,\text{d}x\ \nonumber \\
        &+ 2 \text{i}\int_D k^2 \mathrm{Im}(n)w_1\overline{p_2} \,\text{d}x . \label{aicbc-posT-Kdef}
    \end{align}
    Notice that (\ref{aicbc-posT-Kdef}) contains only $L^2$ terms in $D$ and $B_R$, therefore the compactness of $K$ follows similarly to the first case. 

Now, we wish to show that $\mathrm{Re}(S)$ is a positive coercive operator. To this end, notice that after some calculations we obtain that 
    \begin{align}
        \big\langle \mathrm{Re}(S)(f_1,p_1,h_1)&,(f_2,p_2,h_2)\big\rangle _{X,X^*}  \nonumber  \\
        &= \int_D \overline{f_2}\cdot (I-\mathrm{Re}(A))f_1+p_1\overline{p_2} \,\text{d}x -  \text{i}\int_D \overline{f_2}\cdot \mathrm{Im}(A)\nabla w_1 \,\text{d}x\ + \text{i}\int_D \nabla\overline{w_2}\cdot \mathrm{Im}(A)f_1 \,\text{d}x\ \nonumber \\
        &\quad \quad \quad +\int_{B_R} \mathrm{Re}(A)\nabla w_1\cdot\nabla\overline{w_2}+ w_1\overline{w_2} \,\text{d}x\ -\frac{1}{2}\int_{\partial B_R}w_1\overline{\Lambda_k w_2} + \overline{w_2}\Lambda_k w_1\,\text{d}s\ \label{aicbc-posT-Re(S)def}
    \end{align}
    Notice, by Young's Inequality that
    $$\Big| \int_D \overline{f} \cdot \text{Im}(A)\nabla w \,\text{d}x \Big| \leq \frac{\alpha}{2}(\text{Im}(A)f, f)_{L^2(D)} + \frac{1}{2\alpha}(\text{Im}(A)\nabla w, \nabla w)_{L^2(D)}$$
    for any $\alpha >0$. Therefore, we have that 
    \begin{align*}
        \big\langle \mathrm{Re}(S)(f_1,p_1,h_1),(f_1,p_1,h_1)\big\rangle _{X,X^*}&\geq \Big( \big[ I-\text{Re}(A)-\alpha \text{Im}(A) \big]f_1, f_1 \Big)_{L^2(D)} + (p_1,p_1)_{L^2(D)} \\
        &+ \Big( \big[ \text{Re}(A)-\frac{1}{\alpha} \text{Im}(A) \big]\nabla w_1, \nabla w_1 \Big)_{L^2(D)} + (\nabla w_1,\nabla w_1)_{L^2(B_R\setminus D)} \\
        &+ (w_1,w_1)_{L^2(B_R)} - \int_{\partial B_R} \overline{w_1}\text{Re}(\Lambda_k)w_1 \, \text{d}s .
    \end{align*}
    Provided that $\alpha >0$ is a constant such that $I-\mathrm{Re}(A)-\alpha \mathrm{Im}(A)$ and $\mathrm{Re}(A)-\frac{1}{\alpha}\mathrm{Im}(A)$ are positive definite uniformly in $D$, we have that $\mathrm{Re}(S)$ is a positive coercive operator using the same contradiction argument as in the first part of the proof. 
\end{proof}
Next, we will analyze the behavior of the imaginary part of the operator $T$. Studying the imaginary part of $T$ is necessary for our coercivity result for the operator. We now prove that the imaginary part of the operator $T$ is positive on $\overline{\text{Range}(H)}$ provided that the wave number $k$ is not a transmission eigenvalue. 

Consider the following boundary value problem,
\begin{align}
    \nabla \cdot A\nabla u+k^2n u=0  \quad \text{ and } \quad \Delta v+k^2v=0 \quad&\text{in $D$} \label{aicbc-TEVprob1} \\ 
        u + \text{i} \eta \partial_{\nu_A}u = v  \quad \text{ and } \quad  \partial_\nu v=\partial_{\nu_A}u  \quad &\text{on $\partial D$} \label{aicbc-TEVprob2} 
\end{align}
Here, we let $\partial_{\nu_A} = \nu \cdot A\grad$ on the boundary of the scatterer $D$. The values $k\in\mathbb{C}$ such that there exists a non-trivial solution $(u,v)\in H^1(D)\times H^1(D)$ to (\ref{aicbc-TEVprob1})--(\ref{aicbc-TEVprob2}) are called transmission eigenvalues. This boundary value problem is obtained by assuming that $F$ is not injective with a dense range, meaning the scattered field is equal to zero outside the scatterer $D$. Similar transmission eigenvalue problems have been recently studied in \cite{Victor-TEV,kleefeldcolton17,aniso-cbc-te-xiang, kleefeldpieronek}.

\begin{theorem} \label{aicbc-thm-Im(T)}
    Let the operator $T$ be defined as in (\ref{aicbc_Tdef}). Then, we have that
    $$\mathrm{Im}\big\langle T(f,p,h),(f,p,h)\big\rangle _{X,X^*}>0$$
    for all $(0,0,0)\neq (f,p,h)\in \overline{\mathrm{Range}(H)}$, provided that $k$ is not a transmission eigenvalue. 
\end{theorem}
\begin{proof}
    Recall the following definition of the operator $T$ from the proof of Theorem \ref{aicbc-thm-T=S+K}
    \begin{align*}
        \big\langle T(f,p,h),(f,p,h)\big\rangle _{X,X^*} &= \int_D (\overline{\nabla w+f})\cdot (I-A)(\nabla w +f)+k^2(n-1)|w+p|^2\,\text{d}x\  \nonumber\\
        &+ \int_{\partial D} \text{i} \eta|\partial_\nu w_{+}+h|^2 \,\text{d}s \nonumber\\
        &-\int_{B_R} |\nabla w|^2-k^2|w|^2 \,\text{d}x\ +\int_{\partial B_R}\overline{w}\Lambda_k w\,\text{d}s .
    \end{align*}
    Since $w$ is a radiating solution to the Helmholtz equation outside of $D$ we have that
    \begin{align}
        \int_{\partial B_R}\overline{w}\Lambda_k w\,\text{d}s\ = \int_{\partial B_R}\overline{w}\partial_\nu w\,\text{d}s\ \longrightarrow \text{i}|\gamma|^2k\int_{\mathbb{S}^{d-1}}|w^\infty|^2\,\text{d}s\ \quad\text{as}\quad R\longrightarrow \infty .
    \end{align}
    Therefore, as $R\longrightarrow\infty$, we see that
    \begin{align}
        \text{Im}\big\langle T(f,p,h),(f,p,h)\big\rangle _{X,X^*} &= \int_D (\overline{\nabla w+f})\cdot (-\text{Im}(A))(\nabla w +f)+k^2 \text{Im}(n)|w+p|^2\,\text{d}x\  \nonumber\\
        &+ \int_{\partial D}\eta|\partial_\nu w_{+}+h|^2 \,\text{d}s\ + |\gamma|^2k \int_{\mathbb{S}^{d-1}}|w^\infty|^2\,\text{d}s .  \label{aicbc-Im(T)def}
    \end{align}
    By our assumptions on the imaginary parts of the coefficients, we conclude that the imaginary part of $T$ is non-negative in $X^*$. 

    Next, we prove that the imaginary part of $T$ is positive on $\overline{\text{Range}(H)}$. To this end, assume that there exists $(f,p,h)\in \overline{\text{Range}(H)}$ such that 
    $$\text{Im}\big\langle T(f,p,h),(f,p,h)\big\rangle _{X,X^*}=0,$$
    we will prove that $(f,p,h)=(0,0,0)$. From the definition of $H$, we have that 
    $$(f,p,h)=\Big(\nabla v|_D, v|_D, \partial_\nu v|_{\partial D} \Big)$$
     where $v$ is a solution to the Helmholtz equation in $D$. Notice, from (\ref{aicbc-Im(T)def}), we have that $w^\infty=0$ and by Rellich's Lemma we have that $w=0$ in $\mathbb{R}^d\setminus D$. By the boundary conditions (\ref{directaicbc_aux2}), we obtain that 
    $$w_-+  \text{i}\eta \partial_{\nu_A}(w_-+v)=0 \quad \text{and}\quad \partial_\nu v=\partial_{\nu_A}(w_-+v) \quad \text{ on }\partial D $$
since we have that $\partial_\nu w_+=w_+=0$ on $\partial D$. 
    Now, we get that $(w+v,v)$ satisfies the boundary value problem
    \begin{align*}
        \nabla \cdot A\nabla (w+v)+k^2n(w+v)=0  \quad \text{ and } \quad \Delta v+k^2v=0 \quad&\text{in $D$} \\ 
        (w +v) +  \text{i}\eta \partial_{\nu_A}(w+v)=v  \quad \text{ and } \quad  \partial_\nu v=\partial_{\nu_A}(w+v)  \quad &\text{on $\partial D$} .
    \end{align*}
    By our assumption that $k$ is not a transmission eigenvalue, we have that the above boundary value problem only admits the trivial solution $(w+v,v)=(0,0)$, meaning $(f,p,h)=(0,0,0)$, proving the claim.
\end{proof}
Now that we have shown that our operator $T$ has these properties, the next part of our analysis is the Funk--Hecke integral identity. This identity will allow us to evaluate the Herglotz wave function for the conjugate plane wave $\phi_z=\text{e}^{- \text{i}kz\cdot\hat{y}}$, given by
\begin{align}
    v_{\phi_z}(x)=\int_{\mathbb{S}^{d-1}}\text{e}^{-\text{i}k(z-x)\cdot\hat{y}}\,\text{d}s(\hat{y}) = 
    \begin{cases}
        2\pi J_0(k|x-z|) & \quad \text{in } \mathbb{R}^2, \\
        \\
        4\pi j_0(k|x-z|) & \quad \text{in }  \mathbb{R}^3,
    \end{cases} \label{aicbc-Funk-Hecke}
\end{align}
where $J_0$ is the Bessel function of the first kind of order $0$ and $j_0$ is the spherical Bessel function of the first kind of order $0$. These Bessel functions and their derivatives satisfy
$$|J_0(t)|  \leq \frac{C}{\sqrt{t}} \quad \text{ and }\quad |j_0(t)| \leq \frac{C}{t} \quad \text{for all $t>0$}$$
for some constant $C$. Therefore, we have a bound for the norm of the Funk-Hecke integral identity and for the norm of $H\phi_z$, namely

$$\|v_{\phi_z}\|_{H^1(D)} \leq C\big[ \text{dist}(z,D) \big]^\frac{1-d}{2} $$
which implies that 
$$ \|H\phi_z\|^2_{X^*} = \Big(\|\nabla v_{\phi_z}\|^2_{L^2(D)}+ \|v_{\phi_z}\|^2_{L^2(D)}+\|\partial_\nu v_{\phi_z}\|^2_{H^{-1/2}(\partial D)}\Big) \leq C\big[ \text{dist}(z,D) \big]^{1-d}$$
where dist$(z,D)$ is the distance from $z$ to $D$.

With the factorization of the far--field operator $F$, the fact that $T$ is bounded and coercive, and the Funk--Hecke integral identity, we can solve the inverse problem for recovering $D$ by using the decay of the Bessel functions. The factorization of $F$ gives us that 
$$\big| \big(F\phi_z,\phi_z\big)_{L^2(\mathbb{S}^{d-1})}\big|=\big|\big\langle TH\phi_z,H\phi_z\big\rangle_{X,X^*}\big|.$$
This combined with $T$ being bounded and coercive on $\overline{\text{Range}(H)}$, we get
\begin{align*}
    C_1 \|H\phi_z\|^2_{X^*}\leq \big| \big(F\phi_z,\phi_z\big)_{L^2(\mathbb{S}^{d-1})}\big| \leq C_2 \|H\phi_z\|^2_{X^*}
\end{align*}
for some constants $C_1$ and $C_2$. 
Given the above analysis, we have the following result. 
\begin{theorem}
    For the imaging functional $I(z)= \big| \big(F\phi_z,\phi_z \big)_{L^2(\mathbb{S}^{d-1})}\big|^p$ with $p\in \mathbb{N}$, we have
    \begin{align}
        I(z) = \mathcal{O}\Big(  \mathrm{dist}(z,D)^{p(1-d)} \Big) , \label{aicbc-imagingfunctional-innerprod}
    \end{align}
    where $\mathrm{dist}(z,D)$ is the distance from $z$ to $D$.
\end{theorem}

Given the decay rate of our imaging functional, this functional will allow us to get a reconstruction of the scatterer $D$ using the far--field data. In the next section, we will introduce a direct sampling method that will reconstruct our scatterer using the Cauchy data.

\section{Reconstruction via Cauchy Data} \label{aicbc-section-cauchy}
In this section, we introduce and analyze a new imaging functional. This imaging functional will utilize Cauchy data to reconstruct our scatterer $D$. This functional was first introduced in \cite{sampling-with-deeplearning} for an isotropic material and was explored in \cite{liem-paper-rgf} for an anisotropic material. Here, we analyze the same imaging functional for the case where we have an anisotropic material with a conductive boundary. We define the reciprocity gap imaging functional $I_{RG}(z)$ as
\begin{align}
    I_{RG}(z) := \sum_{j=1}^{N} \Bigg|  \int_{\partial \Omega} u^s(x,\hat{y}_j)\partial_\nu \mathrm{Im}\Phi(x,z)-\partial_\nu u^s(x,\hat{y}_j)\mathrm{Im}\Phi(x,z) \,\text{d}s(x)\  \Bigg|^p  ,  \label{aicbc-imagingfunctional-reciprocity}
\end{align}
where $\overline{D}\subset \Omega$, $p\in \mathbb{N}$ and $\Phi(x,y)$ is the free-space Green's Function defined in \eqref{aicbc-fundamentalsol-def}. Here, we assume that $\hat{y}_j$ for $j = 1,\ldots , N$ are unique incident directions on the unit circle/sphere. As seen in \cite{liem-paper-rgf}, $N$ does not need to be large in order to get an accurate reconstruction of the scatterer. Also, provided that the distance $\Omega = B(0,R)$ for $R \gg 1$ we have that $\partial_\nu u^s(x,\hat{y}_j) \approx \text{i} k u^s(x,\hat{y}_j)$ by the radiation condition \eqref{Sommerfield_aicbc}.

Notice that from the definition 
    \begin{align*}
        \mathrm{Im}\Phi(x,z)= \begin{cases}
        \frac{1}{4} J_0(k|x-z|) & \quad \mathrm{in } \,\,\mathbb{R}^2, \\
        \\
        \frac{k}{4\pi} j_0(k|x-z|) & \quad \mathrm{in }  \,\, \mathbb{R}^3 ,
    \end{cases}
    \end{align*}
 we see that $\mathrm{Im}\Phi(x,z)$ is a solution to the Helmholtz equation in $\R^d$. 
We want this imaging functional to be effective at reconstructing $D$. We analyze the behavior of $I_{RG}(z)$ in the following theorems to show a similar decay rate as the distance of the sampling point $z$ moves away from the scatterer. Similar analysis has been done in \cite{liem2}. In order to develop the resolution analysis for this imaging functional, we first use \eqref{directaicbc1}--\eqref{directaicbc2} to write the integrals over $\partial \Omega$ as integrals over $\partial D$. 

\begin{theorem}
    The imaging functional $I_{RG}(z)$ satisfies
    \begin{align}
        I_{RG}(z) = \sum_{j=1}^{N} \Bigg|  \int_{\partial D} &u^s_-(x,\hat{y}_j) \partial_\nu \mathrm{Im}\Phi (x,z) +\mathrm{i}\eta \partial_{\nu_A} \Big(u^s_-(x,\hat{y}_j)+u^i(x,\hat{y}_j) \Big)\partial_\nu \mathrm{Im}\Phi (x,z)\,\mathrm{d}s(x)\ \nonumber \\
        &+ \int_{\partial D} \mathrm{Im}\Phi (x,z)\partial_\nu u^i(x,\hat{y}_j) - \mathrm{Im}\Phi (x,z)\partial_{\nu_A} \Big(u^s_-(x,\hat{y}_j)+u^i(x,\hat{y}_j) \Big) \,\mathrm{d}s(x)\ \Bigg|^p. \label{RG-funcIntsoverD}
    \end{align} 
\end{theorem}
\begin{proof}
    In order to prove the claim, we use Green's 2nd Identity along with the fact that both $\mathrm{Im}\Phi$ and $u^s$ solve the Helmholtz equation in $\Omega\setminus \overline{D}$. With this, we obtain that 
    \begin{align*}
        0 &= \int_{\Omega\setminus\overline{D}} u^s \Big(\Delta \mathrm{Im}\Phi (\cdot ,z) +k^2\mathrm{Im}\Phi (\cdot ,z) \Big)-\mathrm{Im}\Phi (\cdot ,z) \Big(\Delta u^s +k^2 u^s \Big) \mathrm{d}x \\
        &= \int_{\partial \Omega} u^s\partial_\nu \mathrm{Im}\Phi (\cdot ,z)-\partial_\nu u^s\mathrm{Im}\Phi (\cdot ,z) \,\mathrm{d}s(x)\ - \int_{\partial D} u^s_+\partial_\nu \mathrm{Im}\Phi (\cdot ,z)-\partial_\nu u^s_+ \mathrm{Im}\Phi (\cdot ,z) \,\mathrm{d}s(x).
    \end{align*}
    Therefore, using the boundary conditions 
    $$u^s_+ - u^s_- =  \text{i}\eta (\nu\cdot A\nabla (u^s_-+u^i))  \quad \text{ and } \quad  \partial_\nu (u^s_+ + u^i) = \nu \cdot A\nabla (u^s_-+u^i)  \quad \text{on $\partial D$},$$
    we have that 
    \begin{align*}
        \int_{\partial \Omega} u^s\partial_\nu \mathrm{Im}\Phi (\cdot ,z)-\partial_\nu u^s\mathrm{Im}\Phi (\cdot ,z) \,\mathrm{d}s(x)\ &= \int_{\partial D} u^s_+\partial_\nu \mathrm{Im}\Phi (\cdot ,z)-\partial_\nu u^s_+ \mathrm{Im}\Phi (\cdot ,z) \,\mathrm{d}s(x)\ \\
        &= \int_{\partial D} u^s_- \partial_\nu \mathrm{Im}\Phi (\cdot ,z) + \text{i}\eta \partial_{\nu_A}(u^s_-+u^i)\partial_\nu \mathrm{Im}\Phi (\cdot ,z) \,\mathrm{d}s(x)\  \\
        &+ \int_{\partial D} \mathrm{Im}\Phi (\cdot ,z) \partial_\nu u^i - \mathrm{Im}\Phi (\cdot ,z) \partial_{\nu_A}(u^s_- + u^i) \,\mathrm{d}s(x).
    \end{align*}
    This proves the claim by summing over the incident directions. 
    \end{proof}
This result is useful, because we can now express our imaging functional $I_{RG}(z)$ with data along the boundary of the scatterer $D$ instead of along the boundary of $\Omega$. Given that we have this expression for the imaging functional $I_{RG}(z)$ using data on $\partial D$, we can now analyze the decay rate of this functional.
\begin{theorem}
    For the imaging functional $I_{RG}(z)$, we have
    \begin{align}
        I_{RG}(z) = \mathcal{O}\Big(  \mathrm{dist}(z,D) ^\frac{p(1-d)}{2} \Big) \label{aicbc-OoC-reciprocity}
    \end{align}
    where $\mathrm{dist}(z,D)$ is the distance from $z$ to $D$.
\end{theorem}
\begin{proof}
    Recall the following expression of the reciprocity gap functional on $\partial \Omega$
    \begin{align*}
        \int_{\partial \Omega} u^s\partial_\nu \mathrm{Im}\Phi (\cdot ,z)-\partial_\nu u^s \mathrm{Im}\Phi (\cdot ,z) \,\mathrm{d}s(x)\ &= \int_{\partial D} u^s_- \partial_\nu \mathrm{Im}\Phi (\cdot ,z) + \text{i}\eta \partial_{\nu_A}(u^s_-+u^i)\partial_\nu \mathrm{Im}\Phi (\cdot ,z) \,\mathrm{d}s(x)\  \\
        &+ \int_{\partial D} \mathrm{Im}\Phi (\cdot ,z) \partial_\nu u^i - \mathrm{Im}\Phi (\cdot ,z) \partial_{\nu_A}(u^s_-+u^i) \,\mathrm{d}s(x) .
    \end{align*}
    We want to find the bounds for each of these terms on $\partial D$ to eventually get the decay rate of the $I_{RG}(z)$. First, we note that $\mathrm{Im}\Phi (\cdot ,z)$ and $u^i$ are both smooth solutions to the Helmholtz equation in $\mathbb{R}^d$, also
    $$\partial_\nu \mathrm{Im}\Phi (\cdot ,z) \in H^{1/2}(\partial D)\quad \text{and}\quad \nabla \cdot A\nabla (u^s+u^i)+k^2n(u^s+u^i)=0 \quad \text{in }D . $$
    By appealing to Trace Theorems, we are able to estimate the boundary integrals over $\partial D$ as volume integrals in $D$. We will consider the four different integrals that are used to express $I_{RG}(z)$ from \eqref{RG-funcIntsoverD}. 
    
For the first integral, we use the Trace Theorem, Neumann Trace Theorem, and the well-posedness of \eqref{directaicbc1}--\eqref{directaicbc2} to obtain the first estimate 
   \begin{align*}
        |\textbf{I}| = \Bigg| \int_{\partial D} u^s_- \partial_\nu \mathrm{Im}\Phi (\cdot ,z) \,\mathrm{d}s(x)\ \Bigg| &\leq \|u^s_-\|_{H^{1/2}(\partial D)} \|\partial_\nu \mathrm{Im}\Phi (\cdot ,z)\|_{H^{-1/2}(\partial D)}\\
        &\leq C\|u^s\|_{H^1(D)}\Big( \|\Delta \mathrm{Im}\Phi (\cdot ,z)\|_{L^2(D)} + \| \mathrm{Im}\Phi (\cdot ,z)\|_{H^1(D)}  \Big)  \\
        &\leq C\|u^s\|_{H^1(D)} \|\mathrm{Im}\Phi (\cdot ,z)\|_{H^1(D)} \\
        &\leq C\|u^i\|_{H^1(D)} \|\mathrm{Im}\Phi (\cdot ,z)\|_{H^1(D)}.
    \end{align*}
Notice that we have used the fact that $\mathrm{Im}\Phi (\cdot ,z)$ satisfies the Helmholtz equation. Now, for the second integral we use the fact that $\mathrm{Im}\Phi (\cdot ,z)$ along with the Trace Theorems to obtain the estimate
    \begin{align*}
        |\textbf{II}| = \Bigg| \int_{\partial D}  \text{i}\eta \partial_{\nu_A}(u^s_-+u^i)\partial_\nu &\mathrm{Im}\Phi (\cdot ,z) \,\mathrm{d}s(x)\ \Bigg| \leq C\|\partial_{\nu_A}(u^s_-+u^i)\|_{H^{-1/2}(\partial D)}\|\partial_\nu \mathrm{Im}\Phi (\cdot ,z)\|_{H^{1/2}(\partial D)} \\
        &\leq C\Big( \| \nabla \cdot A\nabla (u^s+u^i) \|_{L^2(D)} + \|u^s+u^i\|_{H^1(D)} \Big) \|\mathrm{Im}\Phi (\cdot ,z)\|_{H^2(D)} \\
        &=  C\Big( \| -k^2n (u^s+u^i) \|_{L^2(D)} + \|u^s+u^i\|_{H^1(D)} \Big) \|\mathrm{Im}\Phi (\cdot ,z)\|_{H^2(D)} \\
        &\leq C\|u^i\|_{H^1(D)} \|\mathrm{Im}\Phi (\cdot ,z)\|_{H^2(D)}.
    \end{align*}
Here, we have used the fact that $n$ and $\eta$ are assumed to be bounded functions. For the third integral, we similarly have that 
    \begin{align*}
        |\textbf{III}| = \Bigg|\int_{\partial D} \mathrm{Im}\Phi (\cdot ,z) \partial_\nu u^i \,\mathrm{d}s(x)\ \Bigg| &\leq \|\partial_\nu u^i\|_{H^{-1/2}(\partial D)}\|\mathrm{Im}\Phi (\cdot ,z)\|_{H^{1/2}(\partial D)} \\
        &\leq C \|u^i\|_{H^1(D)}\|\mathrm{Im}\Phi (\cdot ,z)\|_{H^1(D)}. 
    \end{align*}
    Lastly, the fourth integral can be handled as discussed above to get that 
    \begin{align*}
        |\textbf{IV}| = \Bigg|\int_{\partial D} - \mathrm{Im}\Phi (\cdot ,z) \partial_{\nu_A}(u^s_-+u^i) \,\mathrm{d}s(x)\ \Bigg| &\leq \|\partial_{\nu_A}(u^s_-+u^i)\|_{H^{-1/2}(\partial D)}\|\mathrm{Im}\Phi (\cdot ,z)\|_{H^{1/2}(\partial D)} \\
        &\leq C \|u^i\|_{H^1(D)}\|\mathrm{Im}\Phi (\cdot ,z)\|_{H^1(D)} .
    \end{align*}
Using the fact that $\|u^i\|_{H^1(D)} \leq \|u^i\|_{H^1(\Omega)}$ and that it is bounded independent of $D$, we get that 
    \begin{align*}
        |\textbf{I}| + |\textbf{II}| + |\textbf{III}| + |\textbf{IV}| &\leq C\|u^i\|_{H^1(D)}\|\mathrm{Im}\Phi (\cdot ,z)\|_{H^2(D)} \\
        &\leq C\|\mathrm{Im}\Phi (\cdot ,z)\|_{H^2(D)}    
    \end{align*}
    Now, from \cite{harris-Biharmonic-dsm} we have that 
    $$   \|\mathrm{Im}\Phi (\cdot ,z)\|_{H^2(D)}  \leq C\big[ \text{dist}(z,D) \big]^\frac{1-d}{2} .$$
    We conclude with the estimate 
    \begin{align*}
        I_{RG}(z) = \sum_{j=1}^{N} \Big| \textbf{I} + \textbf{II} + \textbf{III} + \textbf{IV} \Big|^p \leq C\big[ \text{dist}(z,D) \big]^\frac{p(1-d)}{2} ,
    \end{align*}
    proving the claim.
\end{proof}
This theorem tells us that our imaging functional $I_{RG}(z)$ is valid for reconstructing our scatterer $D$ using the Cauchy data via the decay rate of the Bessel functions. In the next section, we will provide some numerical examples for our direct sampling methods using both the far--field data and the Cauchy data, in two dimensions.

\section{Numerical Results} \label{aicbc-section-numerics}
In the previous sections, we have proved that the aforementioned direct sampling methods are useful for reconstructing our unknown scatterer $D$. In this section, we provide some numerical examples for our direct sampling methods to validate our theoretical results. In the next subsection, we will assume that $D = B(0,R)$, $A$ is a constant matrix with $n>0$ and $\eta\in\mathbb{C}$ being constants. In this setting, we can use separation of variables to compute the synthetic data to test our algorithms and to verify the theoretical results. Then, we will derive a system of boundary integral equations to compute the synthetic data where we assume that $D$ is a non-circular region, $A$ is a $2\times 2$ positive definite matrix, $\eta\in\mathbb{C}$, and $n>0$. We then perform reconstructions for non-circular regions using this data with various levels of noise.

\subsection{Seperation of Variables} \label{aicbc-subsection-numerics_sepvars}
In this subsection, we want to provide numerical examples for our direct sampling methods for $D=B(0,R)$ i.e. the ball of radius $R$ centered at $0$. We assume $A=aI$, where the constant $a\in \mathbb{C}$, along with $n>0$ and $\eta\in\mathbb{C}$ being constants. Now, (\ref{directaicbc1})--(\ref{directaicbc2}) can be written as 
\begin{align*}
    \Delta u^s + k^2u^s = 0 \quad \text{in $\mathbb{R}^2 \setminus \overline{D}$} \quad \text{ and } \quad  \Delta u+\frac{n}{a}k^2u=0 \quad &\text{in $D$}  \\ 
    (u^s+u^i)^+ = u^- +  \text{i}\eta a\partial_r u^-  \quad \text{ and } \quad  \partial_r (u^s+u^i)^+ = a\partial_r u^-   \quad &\text{on $\partial D$}
\end{align*}
along with the Sommerfeld radiation condition as $|x| \longrightarrow \infty$ for the scattered field $u^s$. We can use the Jacobi-Anger expansion to express the incident field $u^i=\text{e}^{ \text{i} k x \cdot \hat{y}}$ as
$$u^i(x,\hat{y}) = u^i(r,\theta) = \sum_{p\in \mathbb{Z}} \text{i}^p J_p(kr)\text{e}^{\text{i}p(\theta - \phi)}$$
where $x= r \big(\cos(\theta),  \sin(\theta) \big)$ and $\hat{y} = \big(\cos(\phi), \sin(\phi) \big)$.

To solve for the scattered field $u^s$ and the total field $u$, we make the assumption that they can be written as the following series expansions
\begin{align*}
u^s(x,\hat{y}) = u^s(r,\theta) = \sum_{p\in \mathbb{Z}} \text{i}^p \text{u}^s_p H_p^{(1)}(kr)\text{e}^{\text{i}p(\theta - \phi)} \quad \text{for $r>R$, and } \\
u(x,\hat{y})= u(r,\theta)= \sum_{p\in \mathbb{Z}} \text{i}^p \text{u}_p J_p\Big( k\sqrt{\frac{n}{a}}r\Big)\text{e}^{\text{i}p(\theta - \phi)} \quad \text{for $0\leq r<R$}
\end{align*}
where $H^{(1)}_p$ are Hankel functions of the first kind of order $p$ and $J_p$ are Bessel functions of the first kind of order $p$. Notice, that the above expansions satisfy the PDEs above, and we only need to satisfy the boundary conditions. Here $\text{u}^s_p$ and $\text{u}_p$ can be seen as the coefficients that are to be determined by using the boundary conditions at $r=R$. Therefore, we have that $\text{u}^s_p$ and $\text{u}_p$ satisfy the $2 \times 2$ linear system  
\begin{align*}
     \begin{pmatrix}
        H_p^{(1)}(kR) & -\text{i}\eta k\sqrt{na}J'_p\Big( k\sqrt{\frac{n}{a}}R\Big)-J_p\Big( k\sqrt{\frac{n}{a}}R\Big) \\
        kH_p^{(1)'}(kR) & -k\sqrt{na}J'_p\Big( k\sqrt{\frac{n}{a}}R\Big)
    \end{pmatrix} 
    \begin{pmatrix}
        \text{u}^s_p \\
        \text{u}_p
    \end{pmatrix} 
    =
    \begin{pmatrix}
        -J_p(kR) \\
        -kJ'_p(kR) 
    \end{pmatrix}
\end{align*}
for each $p \in \Z$. We can solve for the coefficients $\text{u}^s_p$ by using Cramer's Rule, we then obtain
   $$ \text{u}^s_p = \frac{\text{det}(M_x)}{\text{det}(M)} \quad \text{ with matrices } \quad M=
    \begin{pmatrix}
        H_p^{(1)}(kR) & -\text{i}\eta k\sqrt{na}J'_p\Big( k\sqrt{\frac{n}{a}}R\Big)-J_p\Big( k\sqrt{\frac{n}{a}}R\Big) \\
        kH_p^{(1)'}(kR) & -k\sqrt{na}J'_p\Big( k\sqrt{\frac{n}{a}}R\Big)
    \end{pmatrix} $$
and 
$$ M_x=
    \begin{pmatrix}
        -J_p(kR) & -\text{i}\eta k\sqrt{na}J'_p\Big( k\sqrt{\frac{n}{a}}R\Big)-J_p\Big( k\sqrt{\frac{n}{a}}R\Big) \\
        -kJ'_p(kR) & -k\sqrt{na}J'_p\Big( k\sqrt{\frac{n}{a}}R\Big)
    \end{pmatrix} .$$

With this computation of $\text{u}^s_p$, we get expressions for the scattered field, its normal derivative (i.e. with respect to $r$), and the far--field pattern. These are given by
\begin{align}
    u^s(x,\hat{y}) = u^s(r,\theta) \approx \sum_{|p|=0}^{15} \text{i}^p \text{u}^s_p H_p^{(1)}(kr)\text{e}^{\text{i}p(\theta - \phi)} \quad \text{and} \quad \partial_r u^s(r,\theta) \approx k \sum_{|p|=0}^{15} \text{i}^p \text{u}^s_p H_p^{(1)'}(kr)\text{e}^{\text{i}p(\theta - \phi)} \label{aicbc-scatt-series} 
\end{align}
and 
\begin{align}
    u^\infty(\hat{x},\hat{y}) = u^\infty(\theta) \approx  \frac{4}{\text{i}} \sum_{|p| = 0}^{15} \text{u}^s_p \text{e}^{\text{i}p(\theta - \phi)}\quad \text{ where } \quad \text{u}^s_p = \frac{\text{det}(M_x)}{\text{det}(M)}. \label{aicbc-farfieldpatterB_R}
\end{align}
In our numerical experiments, we use the above approximations of the scattered field, normal derivative of the scattered field, and the far-field pattern. 

We need an approximation for the far--field operator in order to perform the direct sampling methods with the far--field data. We approximate the integral representation of $F$ with a 64--point Riemann sum which leads to the $64 \times 64$ matrix discretization of the far--field operator given by 
\begin{align*}
    \textbf{F} = \left[ \frac{2\pi}{64} u^\infty (\hat{x}_i ,\hat{y}_j ) \right]_{i,j=1}^{64} .
\end{align*}
In order to test the stability, we will add random noise to the discretized far--field operator $\mathbf{F}$ such that
$$\mathbf{F}_{\delta}=\Big[ \mathbf{F}({i,j}) \big(1+\delta\mathbf{E}({i,j}) \big)\Big]^{64}_{i,j=1}\hspace{.5cm}\text{where}\hspace{.5cm} \| \mathbf{E}\|_2=1.$$
Here, the matrix $\mathbf{E} \in \C^{64 \times 64}$ is taken to have uniformly distributed random entries and $0<\delta < 1$ is the relative noise level added to the data. The entries are initially taken such that the real and imaginary parts are uniformly distributed between $[-1,1]$, then the matrix is normalized.

To get our reconstructions using the far--field data, we begin by discretizing a sampling region $[-2,2]\times[-2,2]$ with $100^2$ equally spaced sampling points. We then move our sampling point $z\in [-2,2]\times[-2,2]$ and plot our imaging functional
\begin{align}
    W(z)=\Big| \overline{\bm{\phi}_z} \cdot \textbf{F}_\delta \bm{\phi}_z \Big|^p\quad\text{ where }\quad \bm{\phi}_z=\Big( \text{e}^{- \text{i}k\hat{x}_{1}\cdot z}, \ldots \, , \text{e}^{- \text{i}k\hat{x}_{64}\cdot z} \Big)^\top ,\label{aicbc-numerics-IPfunctional}
\end{align}
for 
$$\hat{x}_i = \hat{y}_i = (\cos(\theta_i),\sin(\theta_i)) \quad \text{ and} \quad {\theta_i}=2\pi(i-1)/64, \quad \text{ for } i=1,\ldots,64. $$
The contour plot of $W(z)$ over $[-2,2]\times[-2,2]$ will serve as our reconstruction of $D$. 

To perform the direct sampling methods with the Cauchy data, we need approximations for the scattered field and the normal derivative of the scattered field. We approximate the representation of the Cauchy data as $64 \times 64$ matrices given by 
\begin{align*}
    \textbf{us} = \left[ u^s ({x}_i ,\hat{y}_j ) \right]_{i,j=1}^{64} \quad \text{and}\quad \textbf{dus} = \left[ \partial_\nu u^s ({x}_i ,\hat{y}_j ) \right]_{i,j=1}^{64} .
\end{align*}
We test the stability by adding random noise to these operators, given by
\begin{align*}
    \textbf{us}_\delta = \left[ \mathbf{us}({i,j}) \big(1+\delta\mathbf{E}^{(1)}({i,j}) \big) \right]_{i,j=1}^{64} \quad \text{and}\quad \textbf{dus}_\delta = \left[ \mathbf{dus}({i,j})\big(1+\delta\mathbf{E}^{(2)}({i,j}) \big) \right]_{i,j=1}^{64} 
\end{align*}
where $\textbf{E}^{(j)}$ for $j=1,2$ is defined similarly as above. To get our reconstructions using the Cauchy data, we again discretize the rectangular sampling region $[-2,2]\times[-2,2]$ that completely encompasses our region $D=B(0,R)$ i.e. $0<R<2$. Here, the region  $\Omega = B(0,3)$ and the integrals over $\partial \Omega$ are discretized via a 64--point Riemann sum approximation. We then move our sampling point $z$ in the  region $[-2,2]\times[-2,2]$ and plot our imaging functional
\begin{align}
    W_{RG}(z)= \sum_{j=1}^{64}\frac{2\pi}{64}\Bigg|  \textbf{us}_\delta(:,j) \cdot \partial_r J_0(k|\Vec{x}-z|)- \textbf{dus}_\delta(:,j) \cdot J_0(k|\Vec{x}-z|) \Bigg|^p \label{aicbc-numerics-RGfunctional} 
\end{align}
where 
\begin{align*}
    \partial_r J_0(k|\Vec{x}-z|)\Big|_{\partial B(0,3)} = -kJ_1(k|\Vec{x}-z|)\Bigg[ \frac{9-\Vec{x}\cdot z}{3|\Vec{x}-z|} \Bigg] ,
\end{align*}
for $\Vec{x} = (x_1 , \cdots , x_{64})^\top \in \mathbb{R}^{64\times 2}$ with 
$${x}_i = 3 (\cos(\theta_i),\sin(\theta_i))  \quad \text{ and } \quad  \hat{y}_i = (\cos(\theta_i),\sin(\theta_i)) \quad \text{ where } \quad {\theta_i}=2\pi(i-1)/64 . $$
for $i=1,...,64$. This is a discretized approximation of the reciprocity gap functional given by the expression (\ref{aicbc-imagingfunctional-reciprocity}) in 2-dimensions.
\begin{example}
    We provide two reconstructions for $D=B(0,1)$ with $5\%$ noise, i.e. $\delta=0.05$. One reconstruction uses the far--field data and the other uses the Cauchy data. Here, we use the parameters $k = 2\pi$, $a = 1/2-\mathrm{i}$, $n = 2+\mathrm{i}$, and $\eta = 2$.
\end{example}
  \begin{figure}[H]
        \centering 
        \includegraphics[scale=0.50]{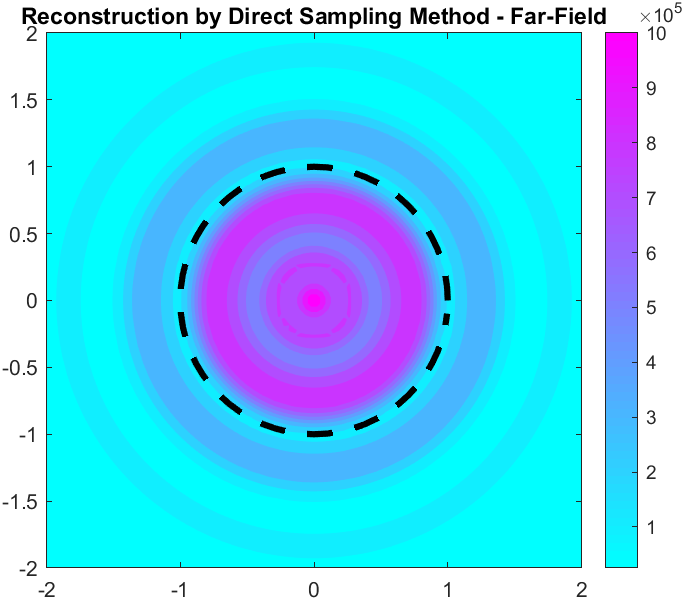} \hspace{0.2in} \includegraphics[scale=.50]{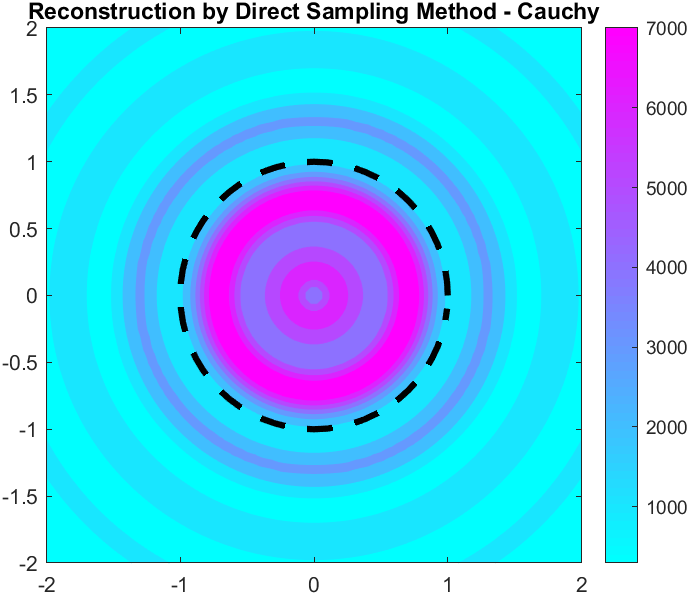}
        \caption{ Reconstruction of circular domain with $5\%$ noise via far--field data with $p=2$ (left) and Cauchy data with $p=3$ (right).}
        \label{aicbc_reconst_B(0,1)}
    \end{figure}
\begin{example}
    We now provide two reconstructions for $D=B(0,3/4)$ with $5\%$ noise, i.e. $\delta=0.05$. One reconstruction uses the far--field data and the other uses the Cauchy data. Here, we use the parameters $k = 4$, $a = 2$, $n = 1/2$, and $\eta = 1$.
\end{example}
  \begin{figure}[H]
        \centering 
        \includegraphics[scale=0.50]{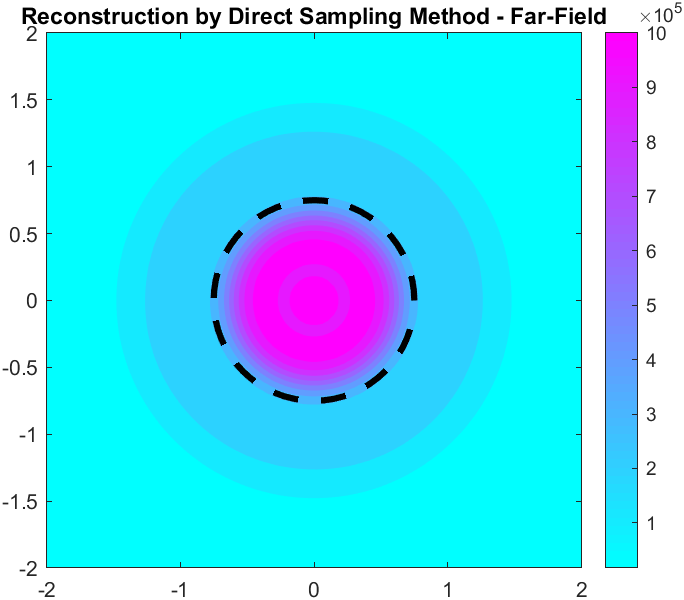} \hspace{0.2in} \includegraphics[scale=.50]{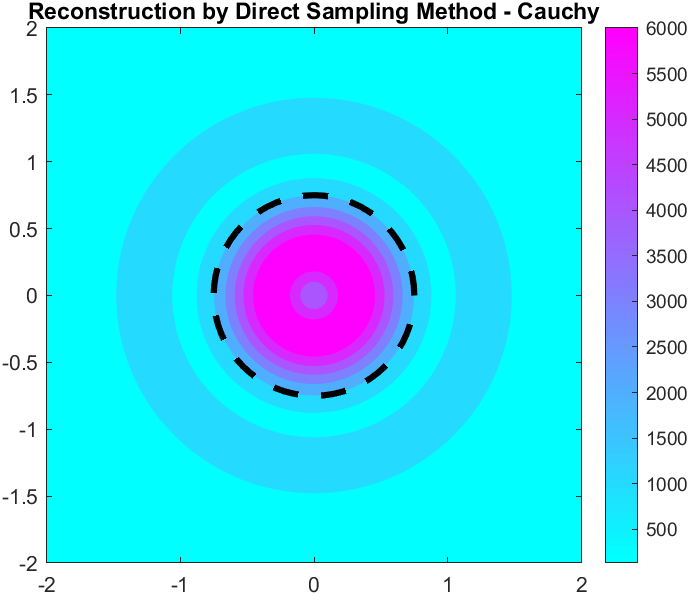}
        \caption{ Reconstruction of circular domain with $5\%$ noise via far--field data with $p=2$ (left) and Cauchy data with $p=3$ (right). }
        \label{aicbc_reconst_B(0,.75)}
    \end{figure}
    
From these examples, we see that our direct sampling imaging functionals are good at recovering the scatterer. Now, we will proceed in the next section by providing more numerical examples for more complex scatterers. Then, we will see the applicability of our methods for different kinds of scatterers.

\subsection{System of Boundary Integral Equations} \label{aicbc-subsection-numerics_boundaryintegral}
In this section, we will provide more examples for our imaging functionals. Here, we will focus on more complex scatterers to test the robustness of our inversion algorithms. To this end, we begin by deriving a system of boundary integral equations for the direct problem which is given by, find $u^s\in H_{\mathrm{loc}}^1(\mathbb{R}^2\backslash \overline{D})$ and $u\in H^1(D)$ satisfying
\begin{align*}
\grad \cdot A\nabla u+k^2 u=0 \quad&\text{in}\quad D\, \quad \text{ and } \quad \Delta u^s+k^2 u^s=0 \quad\text{in}\quad\mathbb{R}^2\backslash \overline{D}\,,\\
\partial_\nu u^s-\partial_{\nu_A}u=-\partial_\nu u^i\quad&\text{on}\quad\partial D\, \quad \text{ and } \quad u^s-u-\mathrm{i}\eta \partial_{\nu_A}u=-u^i \quad\text{on}\quad \partial D\,
\end{align*}
along with the radiation condition where $D$, $A$, $k$, and $\eta$ satisfy the assumptions made in Section \ref{aicbc-section-intro}. For simplicity, we assume that $n=1$ for all the examples in this section. 


We could use the same technique as in \cite{kleefeldcolton17} to derive the corresponding boundary integral equation for the problem above. However, the result would involve the hypersingular boundary integral operator (see \cite[p. 143]{kleefeldcolton17}). Therefore, we use the technique as used in \cite[Section 5.1]{ayalaharriskleefeld2024} and \cite[Section 5.1]{ayalaharriskleefeldex2024} for the isotropic case with a conductive boundary involving one or two parameters, respectively, to the anisotropic case at hand. 
Therefore, we make the ansatz
\begin{align}
 u^s(x)\quad&=\quad\mathrm{SL}_{k}\varphi(x)\,,\quad x\in\mathbb{R}^2 \backslash \overline{D}\,,\label{start1}\\
 u(x)\quad&=\quad\widetilde{\mathrm{SL}}_{k}\psi(x)\,,\qquad x\in D\,,
\label{start2}
\end{align}
where the single layer operators are defined by
\begin{align*}
\mathrm{SL}_{k}\phi(x)\quad&=\quad\int_{\partial D}\Phi(x,y)\phi(y)\,\mathrm{d}s(y) \quad \text{ and } \quad 
\widetilde{\mathrm{SL}}_{k}\phi(x)=\int_{\partial D}\widetilde{\Phi}(x,y)\phi(y)\,\mathrm{d}s(y)\,,\quad  x\notin \partial D\,,
\end{align*}
with $\Phi(x,y)$ the fundamental solution of the Helmholtz equation in two dimensions and 
$$\widetilde{\Phi}(x,y)=\Phi_k(A^{-1/2}x,A^{-1/2}y)/\det\left(A^{1/2}\right)$$ 
the fundamental solution of the modified Helmholtz equation $\grad \cdot A \nabla u+k^2 u=0$ in two dimensions (see \cite[Lemma 2.1]{coltonkressmonk1997}). Using the jump relations \cite[Lemma 2.2]{coltonkressmonk1997}, we have on the boundary 
\begin{align}
    u^s(x)\quad&=\quad\mathrm{S}_{k}\varphi(x)\,,\quad x\in \partial D\,,\label{suppe3}\\
    u(x)\quad&=\quad\widetilde{\mathrm{S}}_{k}\psi(x)\,,\quad x\in \partial D\,,\label{suppe4}
\end{align}
where the single-layer boundary integral operators are defined as
\begin{align*}
    \mathrm{S}_{k}\phi(x)=\int_{\partial D}\Phi(x,y)\phi(y)\,\mathrm{d}s \quad \text{ and } \quad \widetilde{\mathrm{S}}_{k}\phi(x)=\int_{\partial D}\widetilde{\Phi}(x,y)\phi(y)\,\mathrm{d}s\,,\qquad x\in \partial D\,.
\end{align*}
Taking the normal and co-normal derivative of (\ref{start1}) and (\ref{start2}), respectively, we further have on the boundary using the jump relation \cite[Lemma 2.2]{coltonkressmonk1997}
\begin{align}
\partial_{\nu(x)} u^s(x)\quad&=\quad-\frac{1}{2}\varphi(x)+\mathrm{K}'_{k}\varphi(x)\,,\quad x\in \partial D\,,\label{suppe1}\\
    \partial_{\nu_A(x)}u(x)\quad&=\quad\phantom{-}\frac{1}{2}\psi(x)+\widetilde{\mathrm{K}}'_{k}\psi(x)\,,\quad x\in \partial D\,,\label{suppe2}
\end{align}
where the normal and co-normal derivative of the single-layer boundary integral operator are defined as
\begin{align*}
    \mathrm{K}'_{k}\phi(x)=\int_{\partial D}\partial_{\nu(x)}\Phi(x,y)\phi(y)\,\mathrm{d}s \quad \text{ and } \quad 
    \widetilde{\mathrm{K}}'_{k}\phi(x)=\int_{\partial D}\partial_{\nu_A(x)}\widetilde{\Phi}(x,y)\phi(y)\,\mathrm{d}s\,,\qquad x\in \partial D\,.
\end{align*}
Because of the first boundary condition $\partial_\nu u^s-\partial_{\nu_A}u=-\partial_\nu u^i$ and using (\ref{suppe1}) and (\ref{suppe2}), we obtain the first boundary integral equation
\begin{align}
 \left(-\frac{1}{2}I+\mathrm{K}'_{k}\right)\varphi-\left(\frac{1}{2}I+\widetilde{\mathrm{K}}'_{k}\right)\psi=-\partial_\nu u^i\,.
 \label{first}
\end{align}
Because of the second boundary condition $u^s-u-\mathrm{i}\eta \partial_{\nu_A}u=-u^i$ and using (\ref{suppe3}), (\ref{suppe4}), and (\ref{suppe2}) yields
\begin{align}
 \mathrm{S}_{k}\varphi-\widetilde{\mathrm{S}}_{k}\psi-\mathrm{i}\eta \left(\frac{1}{2}I+\widetilde{\mathrm{K}}'_{k}\right)\psi=-u^i\,.
 \label{second}
\end{align}
After we solve the $2\times 2$ system
\begin{align}
\begin{pmatrix}
-\frac{1}{2}I+\mathrm{K}'_{k} \quad \quad \quad \quad \quad \quad -\frac{1}{2}I-\widetilde{\mathrm{K}}'_{k}\\
\quad \mathrm{S}_{k} \quad \quad \quad \quad \quad \quad -\widetilde{\mathrm{S}}_{k}-\mathrm{i}\eta \left(\frac{1}{2}I+\widetilde{\mathrm{K}}'_{k}\right)
\end{pmatrix}
\begin{pmatrix}
\varphi \\
\psi
\end{pmatrix}
=
-\begin{pmatrix}
\partial_\nu u^i \\
u^i
\end{pmatrix}
\label{system}
\end{align}
for $\varphi$ and $\psi$ on $\partial D$, we obtain the scattered field $u^s$ by computing (\ref{start1}). We could also compute the total field $u$ inside $D$ by computing (\ref{start2}).
The far--field pattern of $u^s$ is given by
$$u^\infty(\hat{x})=\mathrm{S}_k^\infty \varphi (\hat{x})\,,$$
where 
\begin{align}
\mathrm{S}^\infty_k\phi(\hat{x})=\int_{\partial D} \mathrm{e}^{-\mathrm{i}k \hat{x}\cdotp y}\phi(y)\,\mathrm{d}s(y)\,,\quad \hat{x}\in \mathbb{S}^1\,.
\label{farfieldexpression}
\end{align}

Now that we have a method for compute the scattered field and far--field pattern, we let $N_f$ be the number of faces used within the boundary element collocation method. Then, the number of collocation nodes is $3\cdotp N_f$. In order to determine the validity of our boundary integral equations, we will compare the calculations with the separation of variable results from the previous section. To this end, for $64$ equidistant incident and observation directions we let ${\bf F}_k\in \mathbb{C}^{64 \times 64}$ the far--field data obtained with the series expansion (\ref{aicbc-farfieldpatterB_R}) and ${\bf F}_k^{(N_f)}\in \mathbb{C}^{64 \times 64}$ be the far--field data obtained via the boundary element collocation method with $N_f$ faces using (\ref{farfieldexpression}). Note, that we make the dependance on the wave number $k$ explicit. The absolute error is defined as
$$\varepsilon_k^{(N_f)} :=\|{\bf F}_k-{\bf F}_k^{(N_f)}\|_2 .$$
Note that the absolute error also depends on the physical parameter $a$ and $\eta$. 

The absolute error of the computed far--field data is shown in Table \ref{tablefarfield} where $D=B(0,2)$ using a disk and the physical parameters $a=3$ and $\eta=1$. The wave numbers are $k=2$, $k=4$, and $k=6$. As we can observe, we obtain very accurate results using $240$ collocation nodes.

\begin{table}[!ht]
\centering
 \begin{tabular}{r|r|r|r|}
  $N_f$ & $\varepsilon_2^{(N_f)}$ & $\varepsilon_4^{(N_f)}$ & $\varepsilon_6^{(N_f)}$ \\
  \hline 
10 (\phantom{1}30) & 0.47427 & 9.22067 & 162.92010\\
20 (\phantom{1}60) & 0.03533 & 0.35907 & 5.32643\\
40 (120) & 0.00405 & 0.00273 & 0.02023\\
80 (240)& 0.00048 & 0.00030 & 0.00078\\
160 (480)& 0.00006 & 0.00004 & 0.00009\\
\hline
 \end{tabular}
 \caption{\label{tablefarfield}Absolute error of the far-field with 64 equidistant incident and observation directions for the disk with radius $R=2$ and the physical parameters $a=3$ and $\eta=1$ for varying number of faces (collocation nodes). The wave numbers are $k=2$, $k=4$, and $k=6$, respectively.}
\end{table}

Likewise, we can compute the scattered field for on $\partial \Omega =\partial B(0,3)$, where the scatterer is again given by $D=B(0,2)$ via the series expansion in (\ref{aicbc-scatt-series}). Just as in the previous section, we compute the scattered field as a $64\times 64$ matrix with $64$ equally spaced points for $x \in \partial  \Omega$ and $\hat{y} \in \mathbb{S}^1$.  As we can observe in Table \ref{tablescatteredfield}, we obtain very accurate results using $240$ collocation nodes.
\begin{table}[!ht]
\centering
 \begin{tabular}{r|r|r|r|}
  $N_f$ & $\varepsilon_2^{(N_f)}$ & $\varepsilon_4^{(N_f)}$ & $\varepsilon_6^{(N_f)}$ \\
  \hline 
10 (\phantom{1}30) &  0.04888 & 0.69465 & 8.75189\\
20 (\phantom{1}60) & 0.00289 & 0.02705 & 0.28886\\
40 (120) & 0.00033 & 0.00016 & 0.00114\\
80 (240)&  0.00004 & 0.00002 & 0.00004\\
160 (480)& 0.00000 & 0.00000 & 0.00000\\
\hline
 \end{tabular}
 \caption{\label{tablescatteredfield}Absolute error of the scattered field with 64 equidistant incident and observation directions on a circle with measurement radius $R_0=3$ for the disk with radius $R=2$ and the physical parameters $a=3$ and $\eta=1$ for varying number of faces (collocation nodes). The wave numbers are $k=2$, $k=4$, and $k=6$, respectively.}
\end{table}

In a similar fashion, we can compute the normal derivative of the scattered field on a given circle with the series expansion (\ref{aicbc-scatt-series}). The computation with the boundary element collocation method is done by taking the normal derivative of (\ref{suppe3}), i.e.
\begin{align*}
\partial_{\nu(x)}u^s(x)=\partial_{\nu(x)}\mathrm{S}_{k}\varphi(x)\,,\quad x\in \partial B(0,3)\,,
\end{align*}
where $\varphi$ is computed by (\ref{system}). Again, we let $\partial \Omega =\partial B(0,3)$ where the scatterer is again given by $D=B(0,2)$ with the normal derivative of the scattered field given as a $64\times 64$ matrix with $64$ equally spaced points for $x \in \partial  \Omega$ and $\hat{y} \in \mathbb{S}^1$. As we can observe in Table \ref{tablendscatteredfield}, we obtain very accurate results using $240$ collocation nodes.
\begin{table}[!ht]
\centering
 \begin{tabular}{r|r|r|r|}
  $N_f$ & $\varepsilon_2^{(N_f)}$ & $\varepsilon_4^{(N_f)}$ & $\varepsilon_6^{(N_f)}$ \\
  \hline 
10 (\phantom{1}30) &  0.06372 & 1.65238 & 40.38222\\
20 (\phantom{1}60) & 0.00575 & 0.06435 & 1.30574\\
40 (120) & 0.00066 & 0.00061 & 0.00478\\
80 (240)&  0.00008 & 0.00007 & 0.00022\\
160 (480)& 0.00001 & 0.00001 & 0.00003\\
\hline
 \end{tabular}
 \caption{\label{tablendscatteredfield}Absolute error of the normal derivative of the scattered field with 64 equidistant incident and observation directions on a circle with measurement radius $R_0=3$ for the disk with radius $R=2$ and the physical parameters $a=3$ and $\eta=1$ for varying number of faces (collocation nodes). The wave numbers are $k=2$, $k=4$, and $k=6$, respectively.}
\end{table}

With this, we see that the our system of boundary integral equations can accurately compute the scattering data. Therefore, we can provide some numerical examples using the data computed via the boundary integral equations as the synthetic data in order to test our new reconstruction algorithms. Just as in the previous section, the imaging functionals are discretized in the exact same way given by \eqref{aicbc-numerics-IPfunctional} and \eqref{aicbc-numerics-RGfunctional}. More precisely, the synthetic data is computed for different obstacles using the boundary element collocation method with $80$ faces (with $240$ collocation nodes). The physical parameters are 
$$\eta=1, \quad k=6, \quad \text{and } \quad A=[4\; \; \; 1\; ; \; 1\; \; \;4].$$ 
The boundaries of the obstacles under consideration are given in polar coordinates as follows: the kite is given by $$(-1.5\cdot\sin(\phi),\cos(\phi)+0.65\cos(2\phi)-0.65)^\top,$$ 
and the peanut is given by 
$$2\sqrt{\frac{1}{2}\sin^2(\phi)+\frac{1}{10}\cos^2(\phi)}\cdotp (\cos(\phi),\sin(\phi))^\top.$$ 
We provide reconstruction of these shapes using the far--field data and Cauchy data with various levels of noise.
\begin{example}
    We provide four reconstructions for the kite shape via the far--field data with $5\%$ noise, $10\%$ noise, $15\%$ noise, and $20\%$ noise, i.e. $\delta=0.05$, $\delta=0.1$, $\delta=0.15$, and $\delta=0.2$, respectively.
\end{example}
\begin{figure}[H]
        \centering 
        \includegraphics[scale=0.50]{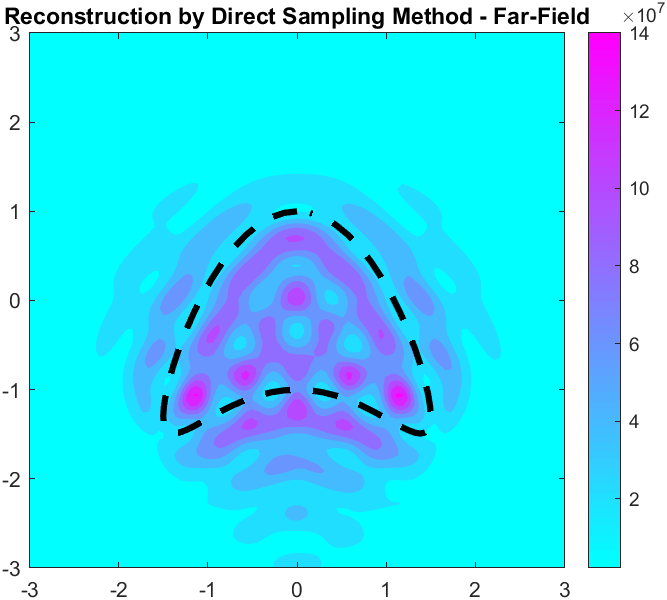} \hspace{0.2in} \includegraphics[scale=.50]{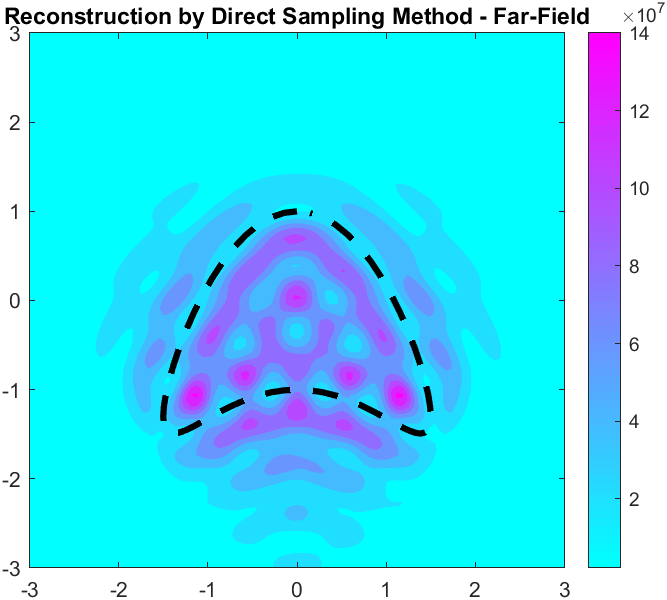}
        \caption{ Reconstruction of the kite shape via far--field data with $5\%$ noise (left) and $10\%$ noise (right), each with $p=2$.}
        \label{aicbc_reconst_kite-5,10}
    \end{figure}
\begin{figure}[H]
        \centering 
        \includegraphics[scale=0.50]{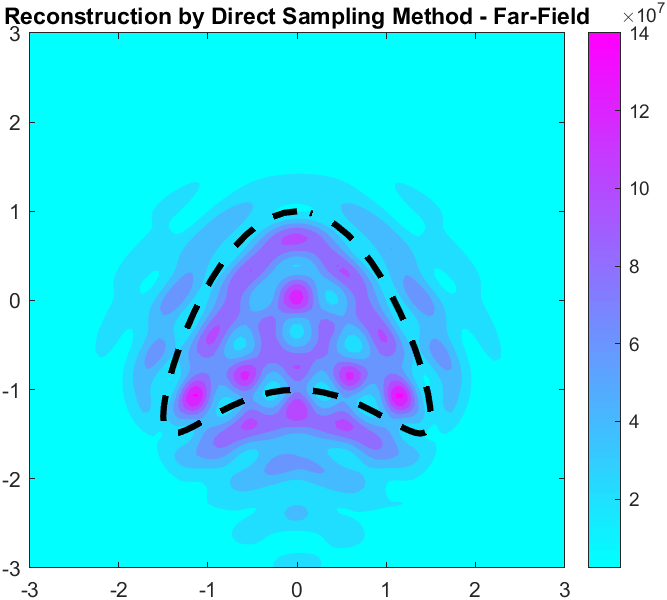} \hspace{0.2in} \includegraphics[scale=.50]{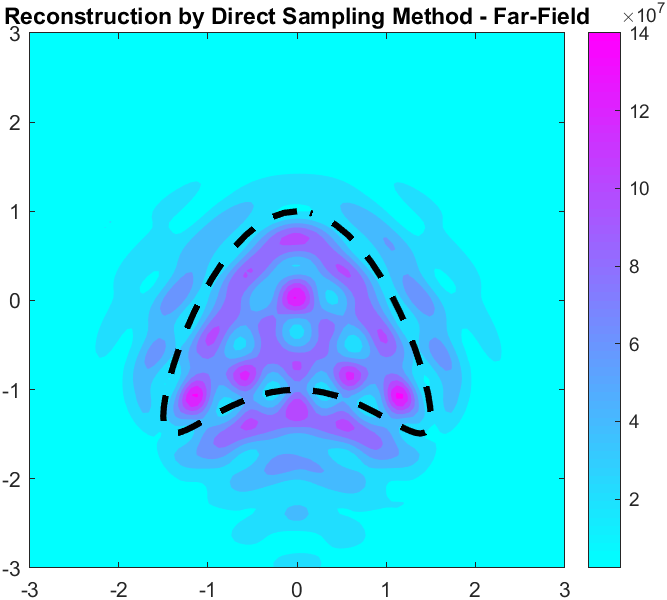}
        \caption{ Reconstruction of the kite shape via far--field data with $15\%$ noise (left) and $20\%$ noise (right), each with $p=2$.}
        \label{aicbc_reconst_kite-15,20}
    \end{figure}
\begin{example}
    We now provide four reconstructions for the peanut shape via the Cauchy data with $5\%$ noise, $10\%$ noise, $15\%$ noise, and $20\%$ noise, i.e. $\delta=0.05$, $\delta=0.1$, $\delta=0.15$, and $\delta=0.2$, respectively.
\end{example}
\begin{figure}[H]
        \centering 
        \includegraphics[scale=0.50]{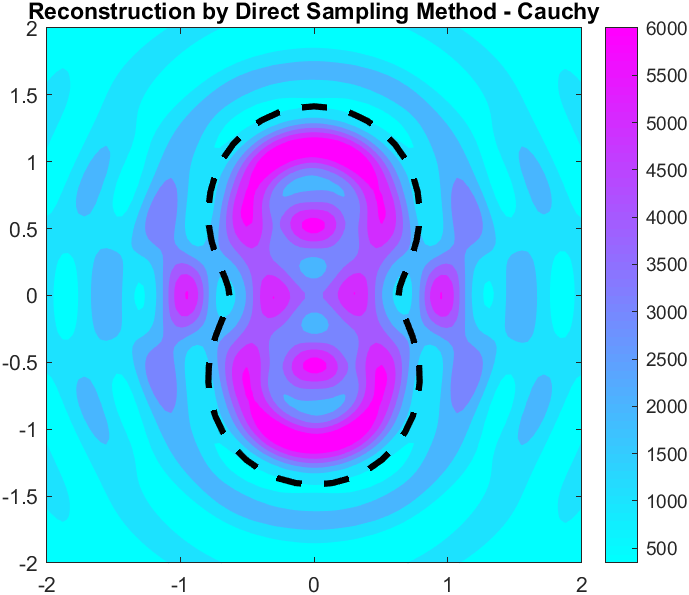} \hspace{0.2in} \includegraphics[scale=.50]{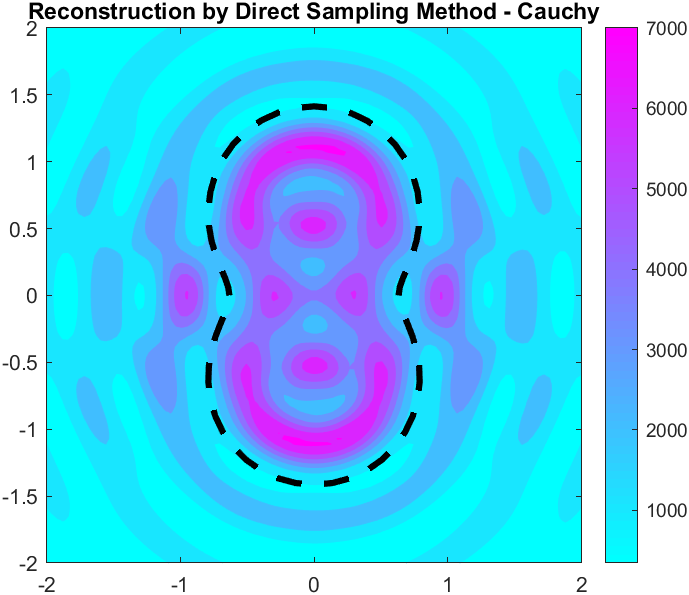}
        \caption{ Reconstruction of the peanut shape via Cauchy data with $5\%$ noise (left) and $10\%$ noise (right), each with $p=3$.}
        \label{aicbc_reconst_peanut-5,10}
    \end{figure}
\begin{figure}[H]
        \centering 
        \includegraphics[scale=0.50]{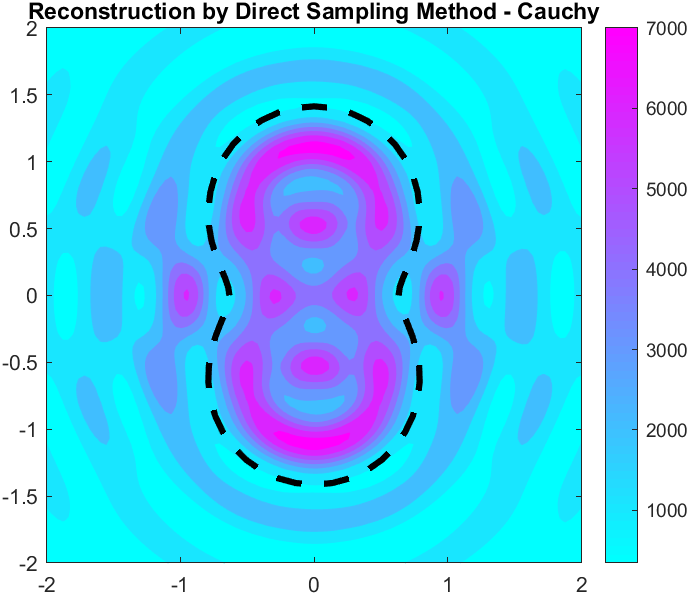} \hspace{0.2in} \includegraphics[scale=.50]{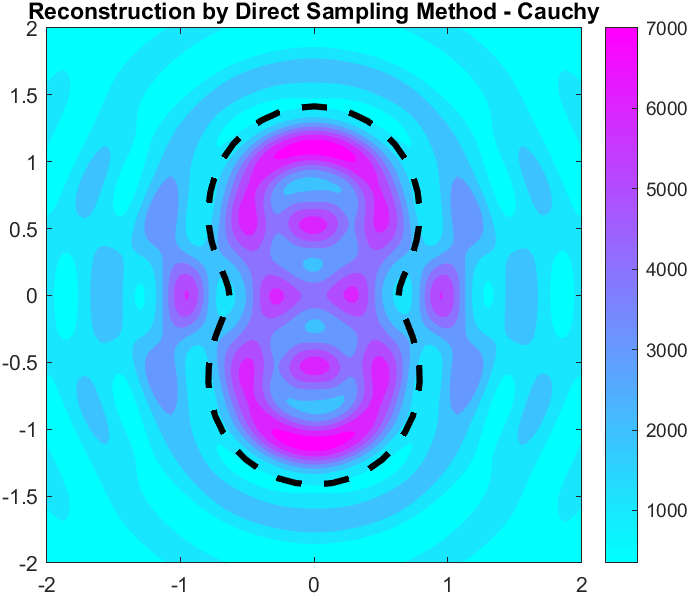}
        \caption{ Reconstruction of the peanut shape via Cauchy data with $15\%$ noise (left) and $20\%$ noise (right), each with $p=3$.}
        \label{aicbc_reconst_peanut-15,20}
    \end{figure}

From our examples, we have faithful reconstructions for the circular regions under either set of assumptions on the coefficients given in Theorem \ref{aicbc-thm-T=S+K}. We also have faithful reconstructions for the non-circular regions, and we see that the method is resilient to added noise in these examples. With this, we see that our direct sampling methods are applicable to reconstructing anisotropic scatterers with a conductive boundary. We have seen that the reconstructions are robust with respect to noisy data which has been proven in \cite{liem-paper-rgf}. 

\section{Conclusions} \label{aicbc-section-conclusion}
In this paper, we studied multiple direct sampling methods for an anisotropic material with a conductive boundary on an unbounded domain. We obtained the necessary symmetric factorization of the far--field operator to apply a direct sampling method via the far--field data. We then derived a direct sampling method using the Cauchy data by studying the reciprocity gap functional for the given material. Our main contribution was adopting the theory of these direct sampling methods for the anisotropic material under two different assumptions on its physical parameters. We then provided some numerical examples of using these methods on different circular and non-circular domains in two dimensions with various levels of noise. We used separation of variables and boundary integral equations to generate the data for these reconstructions. We see that these methods create faithful reconstructions of these domains that are resilient to noise. With that, there is still room for more extensive numerical tests of the direct sampling methods studied in this paper. Future directions for this project are to study the transmission eigenvalue problem \eqref{aicbc-TEVprob1}--\eqref{aicbc-TEVprob2} and the applicability of the direct sampling methods with partial aperture data. \\

\end{document}